\documentclass[pdflatex,sn-mathphys-num]{sn-jnl}
\usepackage{graphicx}
\usepackage{amsmath,amssymb,amsfonts,amsthm}
\usepackage{xcolor,textcomp,booktabs}
\usepackage{algorithm,algorithmic}
\usepackage[title]{appendix}
\theoremstyle{thmstyleone}
\newtheorem{theorem}{Theorem}
\newtheorem{lemma}{Lemma}
\newtheorem{proposition}{Proposition}

\theoremstyle{thmstyletwo}

\theoremstyle{thmstylethree}

\hypersetup{colorlinks=true,linkcolor=blue,citecolor=blue,urlcolor=blue,
 pdftitle={Interval-Constrained Brownian Paths: Exact Interpolation and Extrapolation},
 pdfauthor={Radu Herbei and Kumar Somnath}}

\AtBeginDocument{\allowdisplaybreaks[0]}
\newcommand{\cL}{{\cal L}}
\newcommand{\cF}{{\cal F}}

\newcommand{\sK}{\mathsf{K}}

\newcommand{\sR}{\mathsf{R}}
\newcommand{\sU}{\mathsf{U}}

\newcommand{\EE}{\mathbb {E}}
\newcommand{\PP}{\mathbb {P}}
\newcommand{\RR}{\mathbb {R}}

\newcommand{\Cov}{\operatorname{Cov}}

\newcommand{\bfone}{\mathbf{1}}

\newcommand{\eqd}{\stackrel{\cL}{=}}

\newcommand{\lerh}{\overset{\rm series}{\le}}

\numberwithin{equation}{section}

\begin{document}
\title[Constrained Brownian Interpolation and Extrapolation]{Interval-Constrained Brownian Paths: Exact Interpolation and Extrapolation}
\author*[1]{\fnm{Radu} \sur{Herbei}}\email{herbei.1@osu.edu}
\author[2]{\fnm{Kumar} \sur{Somnath}}\email{kumarsomnath.iitkgp@gmail.com}
\affil*[1]{\orgdiv{Department of Statistics}, \orgname{The Ohio State University}, \orgaddress{\street{1958 Neil Ave.}, \city{Columbus}, \postcode{43210}, \state{OH}, \country{USA}}}
\affil[2]{\orgname{Foursquare Labs, Inc.}, \orgaddress{\city{Seattle}, \state{WA}, \country{USA}}}
\abstract{We study Brownian motion and Brownian bridge processes conditioned to remain in a fixed interval $[0,a]$, focusing on the conditional distribution at a single time. For a Brownian motion in $[0,a]$, conditioned on survival up to time $t$ we recall (and present in a self-contained form) the conditional density of its position at $t$ (exact extrapolation). For a Brownian bridge conditioned to remain in $[0,a]$ on $[0,T]$ we express the probability density at an interior time as a normalized product of killed transition densities (exact interpolation). These densities admit dual complementary series representations, which are linked via the Jacobi theta identity. 
Our main contribution is a unified exact sampling suite for both extrapolation and interpolation that (i) includes boundary endpoints and (ii) automatically switches between density representations to keep acceptance rates efficient across regimes. To that end, we derive simple proposal families which cover both small- and large-time regimes, with an automatic rule selecting the tighter envelope. Our procedures extend to exact simulation of discrete skeleton paths at multiple times via the Markov property.}
\maketitle
\vspace{-24pt}
{\footnotesize\noindent\textbf{Acknowledgements.} We thank the Editors and the Reviewers for their constructive remarks.\par}
\clearpage

\section{Introduction}
\label{sec:intro}

Brownian motion and Brownian bridges under hard constraints form a classical but still practically vital corner of stochastic analysis. This work focuses on the conditional distributions at a single time in two basic situations: a Brownian motion started inside a bounded interval and conditioned to remain there up to a fixed horizon, and a Brownian bridge with prescribed endpoints conditioned to remain inside the same interval. Our stated goal is to develop exact sampling algorithms for these two conditional distributions. 
Their analysis reaches back to the method of images and to spectral expansions for the heat equation with Dirichlet boundary conditions, and modern treatments place them within the broader frameworks of excursion theory, meanders, and killed semigroups; see, among others,  \cite{ito2012diffusion,revuz2013continuous,borodin2012handbook}. The appeal of these conditional laws is both conceptual and practical. Even when interest centers on a single observation time, access to the correct conditional law, rather than a discretization surrogate, can be critical
for likelihood evaluation, data augmentation, and proposal design in Monte Carlo schemes used downstream.

The probability laws studied in this paper admit two complementary representations, both playing a central role. The first uses the reflection perspective to express constrained transition probabilities as rapidly convergent sums. The second uses the eigen-expansion of the Dirichlet heat kernel, yielding sine-series representations that naturally expose the relevant time and length scales. These series are linked, through Jacobi theta identities \cite{NISTHandbook2010}, to well-studied special functions and to distributional facts about meanders, excursions, and bridges; see \cite{revuz2013continuous,borodin2012handbook}. In particular, the law of a Brownian bridge constrained by its supremum relates to the Kolmogorov distribution that underpins the classical goodness-of-fit statistic \cite{Kolmogorov1933,Smirnov1939,Darling1957}, while the meander provides a canonical model for survival from a boundary and admits explicit descriptions of its marginals \cite{durrett1977functionals,Imhof:1984:DFB,BianeYor2008}. For interpolation inside a constrained bridge, a standard factorization expresses the marginal at an interior time as a product, properly normalized, of two killed heat kernels evolving from the left and right endpoints, a perspective that is both probabilistic and analytic.

Although these survival-conditioned probability laws admit classical series representations, turning them into exact samplers is nontrivial. We combine the reflection and spectral perspectives and build exact simulation algorithms based on Devroye's \emph{series method}, which provides an exact accept--reject framework whenever the target density admits alternating or absolutely convergent series with computable monotone bounds \cite{devroye1981series, devroye1986book}. Related exact simulation frameworks for more general diffusions and diffusion bridges include the Exact Algorithm and retrospective exact simulation \cite{EA1,EA2}. We develop simple proposal families whose role is to upper-bound the unnormalized conditional densities across all regimes. The choice between proposals is automatic and based on a small number of scale-free summaries that signal which representation delivers the tighter envelope in the parameter region at hand. The series method ensures that every draw terminates after a finite and data-dependent number of terms while preserving exactness. The algorithms we develop cover interior and boundary endpoints in a unified way, and naturally scale to joint sampling at several time-points by iterating the one-time step and exploiting the Markov property. Throughout, we emphasize proposals and comparisons that lead to code of minimal complexity, with choices that reflect practical constraints faced by users who need dependable samplers rather than delicate, regime-specific implementations.

The manuscript is organized as follows. In Section \ref{sec:prel} we introduce notation and some preliminary facts. Section \ref{sec:sampling} develops and illustrates the exact sampling algorithms, both for a constrained Brownian motion (extrapolation) and a constrained Brownian bridge (interpolation). This section ends with a few computational remarks. Our concluding remarks are in Section \ref{sec:concl}. Detailed proofs of our results are included in the Appendix.

\section{Preliminaries}
\label{sec:prel}

Throughout this paper $a> 0$ is a fixed constant and $x,z\in [0,a]$ are  fixed points. We use $\eqd$ to denote ``equal in distribution''. The arrows $\nearrow$ and $\searrow$ indicate monotone convergence, increasing and decreasing respectively. The symbols $\sU, \sU', \sU'', \ldots$ are used to denote iid ${\rm Uniform}(0,1)$ variates. Let $\{W_t,\ 0\le t\le T\}$ be a standard Brownian motion process on $[0,T]$, defined on some probability space $(\Omega, \cF, \PP)$. The process $\{W^x_t, 0\le t\le T\} = \{x + W_t, 0\le t\le T\}$ is called a Brownian motion started at $x$. When we ``force'' this process to end at the fixed point $z$, we obtain what is known as a Brownian bridge from $x$ to $z$, which we will denote as $\{W^{x\to z}_t, 0\le t\le T\}$. Formally, the bridge can be defined as
$$
W^{x\to z}_t = z\frac{t}{T} + \left(W^x_t - \frac{t}{T}W^x_T\right),\qquad 0\le t\le T\ ,
$$
or, succinctly, we say $\{W^{x\to z}_t, 0\le t\le T\}\eqd \{W^x_t, 0\le t\le T\}\mid [W^x_T = z]$. Details on the careful definition of these processes can be found in standard textbooks, for example \cite{revuz2013continuous}. Consider the following two events
\begin{align*}
A_u(t) &= \left[0\le \inf_{0\le s\le t}W^{u}_s \le \sup_{0\le s\le t} W^{u}_s \le a \right]\ ;\\
B &= \left[0\le \inf_{0\le s\le T}W^{x\to z}_s \le \sup_{0\le s\le T} W^{x\to z}_s \le a \right]\ . 
\end{align*}

We note that $A_x(t)$ and $B$ are written using the compact interval $[0,a]$ and we allow $x,z\in[0,a]$.
Let $D=(0,a)$ and define the exit time
$
\tau := \inf\{s\ge 0:\, W_s^x\notin D\}=\inf\{s\ge 0:\, W_s^x\in\{0,a\}\}.
$
When $x,z\in(0,a)$, the events $A_x(t)$ and $\{\tau>t\}$ (and similarly $B$ and $\{\tau>T\}$ for the
bridge) differ only on the event that the path touches $\{0,a\}$ without ever leaving $[0,a]$, which has
probability $0$. Hence, for interior endpoints we may work on the open interval $(0,a)$ without loss of
generality (as we do below). We treat the cases $x,z\in\{0,a\}$ separately, by limits from the interior; see Section~\ref{sec:sampling}.

The event $A_x(t)$ states that $\{W_s^x, 0\le s\le t\}$ has not left the interval $[0,a]$, while $B$ states that the Brownian bridge $\{W_s^{x\to z}\}$ has not left the interval $[0,a]$. The goal of this paper is two-fold: (i) develop an exact sampling algorithm for the variable $W_t^x$ conditionally on $A_x(t)$ and (ii) develop an exact sampling algorithm for $W_t^{x\to z}$ conditionally on $B$. The challenge stems from the fact that both conditional distributions admit densities (with respect to the Lebesgue measure) which are represented as infinite series, and thus cannot be ``handled'' in the usual way. Expressions for the density 
$
\PP\big(W_t^x \in dy \mid A_x(t)\big)
$
for $t>0$ and $y\in (0,a)$ already exist in the literature, or can be derived using standard techniques, as we show later in Section \ref{sec:sampling}. 
Once this density is available, the Markov nature of the processes under study allows us to write the un-normalized expression for
$\PP(W_t^{x\to z} \in dy \mid B)$, as given in the next result.

\begin{proposition}
\label{prop:markov}
Given $a>0$, $t\in (0,T)$ and $x,y,z\in (0,a)$ we have the following density factorization:
\begin{equation}
    \frac{\PP(W_t^{x\to z}\in dy \mid B)}{dy} \propto \frac{\PP(W_t^x\in dy\mid A_x(t))}{dy}\times \frac{\PP(W_{T-t}^z\in dy\mid A_z(T-t))}{dy}\ .
    \label{eq:factorization}
\end{equation}
\end{proposition}

\begin{proof}
    See Appendix \ref{proof:markov}, also \cite{ito2012diffusion,borodin2012handbook,revuz2013continuous}.
\end{proof}

Before proceeding, we present the following technical result, which gives a set of tight bounds on two absolutely convergent series. These series will appear later on, in the expressions defining the densities of interest.
\begin{proposition}
\label{prop:h1}
Fix $\alpha, \beta>0$ and consider the following absolutely converging series
$$
S^{(1)}(\alpha) = \sum_{n=1}^\infty n^2\exp\{-\alpha n^2\} <\infty\ ,
\qquad
S^{(2)}(\beta) = \sum_{n= 1}^\infty \exp\{-\beta n^2\} < \infty\ .
$$
For $N\ge 1$, define the sequences of partial sums and the remainder terms as
\begin{align*}
S^{(1)}_N(\alpha) &= \sum_{n=1}^{N} n^2\exp\{-\alpha n^2\}, &
R^{(1)}_{N}(\alpha) &= \sum_{n = N+1}^\infty n^2\exp\{-\alpha n^2\},\\
S^{(2)}_N(\beta) &= \sum_{n=1}^{N} \exp\{-\beta n^2\}, &
R^{(2)}_{N}(\beta) &= \sum_{n = N+1}^\infty \exp\{-\beta n^2\}.
\end{align*}
Assume $\rho\in(0,1)$ is given. Then, for
$$
N > N_0(\alpha, \beta, \rho) \stackrel{\rm def}{=} \max\left\{ \frac{1}{2}\left(\frac{1}{\alpha}\log\left(\frac{4}{\rho}\right)-1\right)\ ,\ \frac{1}{2}\left(\frac{1}{\beta}\log\left(\frac{1}{\rho}\right)-1\right) \right\}
$$
the remainder terms are bounded as follows,
\begin{align*}
    R^{(1)}_{N}(\alpha) &\le x^{(1)}_N(\alpha) \stackrel{\rm def}{=} \frac{(N+1)^2\exp\{-\alpha (N+1)^2\}}{(1-\rho)};\\
    R^{(2)}_{N}(\beta) &\le x^{(2)}_N(\beta) \stackrel{\rm def}{=} \frac{\exp\{-\beta (N+1)^2\}}{(1-\rho)};
\end{align*}
and, the following limits hold as $N\to \infty$:
\begin{align*}
    S^{(1)}_N(\alpha) \nearrow S^{(1)}(\alpha)
    \qquad\mbox{ and }\qquad
    S^{(1)}_N(\alpha) + x^{(1)}_N(\alpha) \searrow S^{(1)}(\alpha) \ ,\\
    S^{(2)}_N(\beta) \nearrow S^{(2)}(\beta)
    \qquad\mbox{ and }\qquad
    S^{(2)}_N(\beta) + x^{(2)}_N(\beta) \searrow S^{(2)}(\beta)\ .
\end{align*} 
\end{proposition}
\begin{proof}
    See Appendix \ref{proof:h1}.
\end{proof}

Above, the fixed value $\rho$ can be viewed as a tuning knob. Note that $S^{(1)}_N$ and $S^{(2)}_N$ are decreasing in $\rho$, since $N$ is decreasing with $\rho$ while $x^{(1)}_N$ and $x^{(2)}_N$ are increasing in $\rho$. Thus, in order to achieve tight bounds for $S^{(1)}$ and $S^{(2)}$, one would need to balance the two components. In our simulations we find that an ``average'' value $\rho=1/2$ is sufficient. 
Also, we note that it is possible that $N_0(\alpha, \beta, 1/2) < 0$, for some values of $\alpha, \beta>0$. To avoid this, we define 
\begin{equation}
\label{eq:N0}
N_0 = N_0(\alpha,\beta) = 
\max\left\{\left\lceil N_0(\alpha,\beta, 1/2)\right\rceil+1\ ,\ 1\right\}
\end{equation}
where $\lceil\cdot\rceil$ is the $\mathtt{ceiling}$ function. Then, Proposition \ref{prop:h1} with $\rho=1/2$ gives that
\begin{align}
\sum_{n=1}^{N_0+1} n^2\exp\{-\alpha n^2\} &< S^{(1)}(\alpha) < \sum_{n=1}^{N_0+1} n^2\exp\{-\alpha n^2\} + 2(N_0+1)^2\exp\{-\alpha (N_0+1)^2\}
\label{eq:bound S1}\\
\sum_{n=1}^{N_0+1} \exp\{-\beta n^2\} &< S^{(2)}(\beta) < \sum_{n=1}^{N_0+1} \exp\{-\beta n^2\} + 2\exp\{-\beta (N_0+1)^2\}
\label{eq:bound:S2}
\end{align}

We conclude this section with a reminder that our main sampling tool is \emph{rejection sampling - the series method} \cite{devroye1981series}. Thus, our algorithms are exact in the statistical sense, in that they yield outputs having the correct distribution. We do acknowledge that \emph{any} computer implementation will suffer from a small numerical error, which is due entirely to the fact that computers have finite numerical precision. As such, algebraic calculations are limited to a pre-determined numerical tolerance. We also work under the assumption that unlimited iid $\sU, \sU', \sU'',\ldots \sim {\rm Uniform}(0,1)$ random variates can be obtained easily, at essentially zero computing cost, and we ignore the fact that such draws are pseudorandom.

Devroye's series method is described extensively in Chapter IV, Section 5 of \cite{devroye1986book}. In short, if one wishes to sample from a density $k(y) = \sum_{n\ge 0} a_n(y)$ satisfying the rejection sampling inequality $k(y) \le Mq(y)$, the accept/reject step requires one to decide if $\sU \times M q(Y)\le k(Y)$ or not, for some $Y\sim q(\cdot)$. This can be done \emph{exactly}, without the need to evaluate $k(Y)$ if $k$ can be approximated from above and below by monotone sequences $L_n(y)\nearrow k(y) \swarrow U_n(y)$, as $n\to \infty$. We assume that $L_n$ and $U_n$ have much simpler expressions and can be evaluated exactly. Select $N$ to be large enough such that either $\sU\times Mq(Y) < L_N(Y) < k(Y)$ or $\sU\times Mq(Y) > U_N(Y) > k(Y)$. In light of the convergence assumption above, such an $N$ can always be found. The decision to accept/reject $Y$ can now be made without the need to evaluate $k(Y)$, but simply evaluate $L_N(Y)$ and $U_N(Y)$. Henceforth, we will use $\lerh$ to denote this comparison strategy and testing whether
$$
\sU\times Mq(Y) \lerh k(Y)
$$
is interpreted to mean that if a series appears in the above inequality, it is approximated via monotonically convergent approximating sequences (from below and above) and the comparison is done by evaluating the approximating sequences (exactly) as explained above. We also note that it may be possible that both the right-hand side and the left-hand side above are expressed as series, with corresponding monotonically approximating sequences, from above and below. In this case the strategy can be easily extended and the decision is made in a similar fashion.

\section{Exact Interpolation and Extrapolation under Survival: Algorithms and Implementation}
\label{sec:sampling}

\subsection{Extrapolation under survival}
\label{sec:extrap}

We now tackle the first goal: exact sampling of $W_t^x$ conditionally on $A_x(t)$. In other words, we are interested in simulating the Brownian motion $W^x_t$ at a future time $t$, conditionally on it never leaving the interval $[0,a]$. This is known as \emph{extrapolation}. We distinguish two cases, based on the starting value of the Brownian motion process: (i) $0<x<a$ and (ii) $x=0$. The case $x = a$ is similar to $x=0$, via a symmetry argument. 

\textbf{Case I:} $0<x<a$. It is well known \cite{devroye2010exact, borodin2012handbook,ito2012diffusion,revuz2013continuous,herrmann2020exact} that, for $0<t<T$,
$$
\PP(W_t^x\in dy \mid A_x(t))  \propto g(y;t,x,a)dy
$$
where
\begin{align}
g(y;t,x,a)&=\sum_{n=-\infty}^\infty \frac{1}{\sqrt{2\pi t}}\left(  \exp\left\{ - \frac{(y + 2na- x)^2}{2t} \right\} -  \exp\left\{ - \frac{(y + 2na+ x)^2}{2t} \right\}\right) \nonumber\\
&= \frac{2}{a}\sum_{n=1}^\infty 
\sin\left(\frac{\pi n y}{a}\right)
\sin\left(\frac{\pi n x}{a}\right)
\exp\left\{-\frac{n^2\pi^2}{2a^2}t\right\}
\label{eq:g2}
\end{align}
for $0<y<a$. Sampling a variate with a density proportional to $g$ has been studied in \cite{devroye2010exact, herrmann2020exact} and we will not reproduce all the details here. Both approaches use rejection sampling technique (the series method), albeit with slightly different bounds. Python code implementing Devroye's algorithm is available in the project repository cited in the Code availability statement.
Figure \ref{fig:draws:g} illustrates our simulation results. We set $x=0.5$ in both panels and use $(a,t)=(1.8,1)$ in the left panel and $(a,t)=(2.5,0.2)$ in the right panel. These choices illustrate the two sampling regimes $a/\sqrt{t}<2$ and $a/\sqrt{t}>2$, respectively. Each panel shows a histogram of $10,000$ draws (grey), a kernel density estimate (blue curve), and the theoretical density (red curve). For illustration purposes, the normalizing constant is evaluated numerically.

\begin{figure}[tbp]
    \centering
    \includegraphics[width=\linewidth]{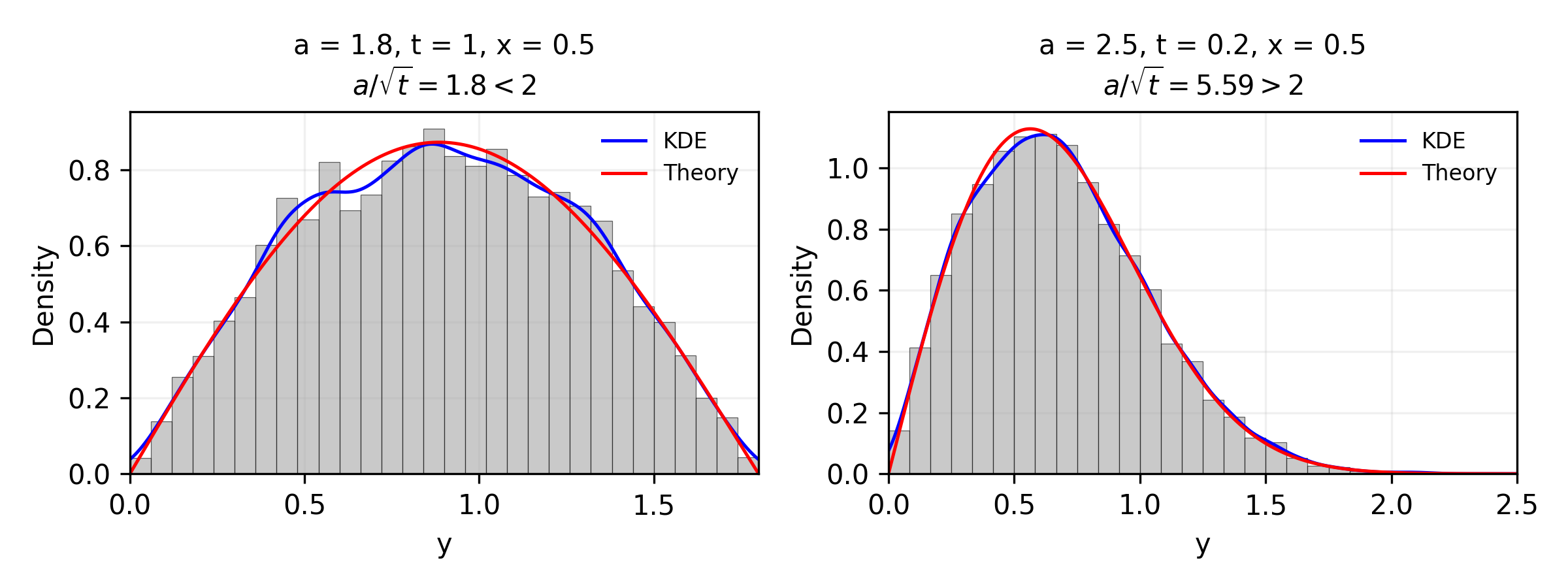}
    \caption{Histograms (grey) of $10,000$ draws from the density proportional to $g(\cdot;t,x,a)$, with kernel density estimates (blue curves) and numerically normalized theoretical densities (red curves). The draws are generated using the algorithm of \cite{devroye2010exact}. We set $x=0.5$ in both panels, with $(a,t)=(1.8,1)$ on the left and $(a,t)=(2.5,0.2)$ on the right, illustrating the regimes $a/\sqrt{t}<2$ and $a/\sqrt{t}>2$, respectively.}
    \label{fig:draws:g}
\end{figure}

In addition, we state the following result, which is proved in \cite{devroye2010exact}, as it will become relevant later.
\begin{proposition} \cite{devroye2010exact}
    \label{prop:bound:g}
    Define $\delta = 4\exp\{-3\pi^2/8\}$. With $g$ being the un-normalized density defined above, the following two bounds hold, for all $y\in [0,a]$:
    \begin{enumerate}
        \item If $a/\sqrt{t} < 2$, then
        \begin{align}
            \label{eq:bound:g:1}
            g(y;t,x,a) &\le \frac{\pi^2 x}{a(1-\delta)} \exp\left\{-\frac{\pi^2t}{2a^2}\right\}  \times \frac{2y}{a^2}\bfone\{0\le y\le a\} \nonumber \\
        &\le \frac{2\pi^2x}{a^2(1-\delta)}\exp\left\{-\frac{\pi^2t}{2a^2}\right\} 
        \stackrel{\rm def}{=} C^{(1)}_g(t,x,a)
        \end{align}
        \item If $a/\sqrt{t}\ge 2$, then 
        \begin{align}
            g(y;t,x,a)&\le \frac{1}{\sqrt{2\pi t }}\left(  \exp\left\{ - \frac{(y - x)^2}{2t} \right\} -  \exp\left\{ - \frac{(y + x)^2}{2t} \right\}\right)\bfone\{0<y<a\} \nonumber\\
            &\le \frac{1}{\sqrt{2\pi t }}\left(  1 -  \exp\left\{ - \frac{(a + x)^2}{2t} \right\}\right) \stackrel{\rm def}{=} C^{(2)}_g(t,x,a)\ .
            \label{eq:bound:g:2}
        \end{align}
    \end{enumerate}
\end{proposition}
\begin{proof}
    See \cite{devroye2010exact}.
\end{proof}

\medskip
\textbf{Case II:} $0=x<a$. The process $\{W^0_t\, , t\ge 0\}$, conditioned to stay in $[0,a]$ has the same law as a Brownian meander, restricted to the interval $[0,a]$ \cite{Imhof:1984:DFB,BianeYor2008}. The conditional density of $W^0_t$ given $A_0(t)$ can be found in several ways. One approach is to carefully take a limit as $x\to 0^+$ in \eqref{eq:g2}, taking into account that the normalizing constant, which depends on $x$, is missing. Alternatively, say $\{W^{\rm me}_s\,,\,0\le s\le t\}$ is a Brownian meander process.
\cite{durrett1977functionals} give the following joint probability:
$$
\begin{aligned}
&\PP\left(W^{\rm me}_t \le y\,,\, \sup_{0\le s\le t}W^{\rm me}_s \le a\right) \\
&\qquad = \sum_{n=-\infty}^\infty \exp\left\{ \frac{-2n^2a^2}{t}\right\} - \exp\left\{ - \frac{(2na + y)^2}{2t} \right\},\ 0\le y\le a \ .
\end{aligned}
$$
From here, after taking a derivative with respect to $y$, we find
\begin{align}
\PP(W^0_t\in dy\mid A_0(t)) &= \PP\left(W^{\rm me}_t\in dy \biggm| \sup_{0\le s\le t}W^{\rm me}_s \le a\right) dy\nonumber \\
&= \frac{\displaystyle \sum_{n=-\infty}^\infty \left(\frac{2na + y}{t}\right) \exp\left\{ - \frac{(2na + y)^2}{2t} \right\}  }{\PP\left(\sup_{0\le s\le t}W^{\rm me}_s\le a\right)}dy
\label{eq:h0}
\end{align}
For now, we will use the fact that this density is proportional to 
\begin{equation}
\label{eq:h1}
h(y;t,a) = \sum_{n=-\infty}^\infty (2na+y)\exp\left\{-\frac{(2na+y)^2}{2t}\right\}\ ,\ 0\le y\le a\ .
\end{equation}
Using the Jacobi theta identity \cite{NISTHandbook2010} we find the alternative representation
\begin{align}
\sum_{n=-\infty}^\infty (2na+y) \exp\left( -\frac{(2na+y)^2}{2t} \right)
&=
\frac{\sqrt{2}\pi^{3/2}t^{3/2}}{a^2} \sum_{n=1}^\infty n  \sin\left( \frac{\pi n y}{a} \right) \exp\left( -\frac{\pi^2  n^2}{2a^2} t \right) \nonumber \\
&= K^h \sum_{n=1}^\infty n  \sin\left( \frac{\pi n y}{a} \right) \exp\left( -\frac{\pi^2  n^2}{2a^2} t \right)\ ,
\label{eq:h2}
\end{align}
where the definition of the constant $K^h$ is clear from the display above. Next, we establish a suitable bound for $h(y;t,a)$ to be used in a rejection sampling algorithm. Denote
\begin{equation}
\label{eq:alpha:beta}
\alpha_t \stackrel{\rm def}{=} \frac{\pi^2 t}{2a^2}\ , \qquad \beta_t \stackrel{\rm def}{=} \frac{a^2}{2t} = \frac{\pi^2}{4\alpha_t}\ ,
\end{equation}
and observe that $\alpha_t$ directly affects how fast the exponential terms appearing in \eqref{eq:h1} and \eqref{eq:h2} are decaying. 
For large values of $\alpha_t$ the series in \eqref{eq:h2} will converge faster, while for small $\alpha_t$, the series \eqref{eq:h1} will result in a better bound. The next result establishes two alternative dominating envelopes for the function $h$.

\begin{theorem} 
\label{thm:bound:h}
With $\alpha_t, \beta_t$ defined above, the following two bounds hold for any $t>0$ and $y\in [0,a]$,
\begin{equation}
    \label{eq:bound:h1}
    h(y;t,a) \le  y\exp\left\{-\frac{y^2}{2t}\right\}+ y\displaystyle{\sum_{n=1}^{\infty}}\exp\left\{-\beta_t n^2\right\}
\end{equation}
and
\begin{equation}
    \label{eq:bound:h2}
    h(y;t,a) \le 4\pi^{-1/2}\alpha_t^{3/2} \sum_{n=1}^\infty n^2\exp\{-\alpha_t n^2\} \times \min\{ y,a-y\}\ .
\end{equation}
\end{theorem}
\begin{proof}
    See Appendix \ref{proof:bound:h}.
\end{proof}
Since we aim to develop a rejection sampler for $h$, the choice between \eqref{eq:bound:h1} and \eqref{eq:bound:h2} is made by comparing
\begin{align*}
I_1(t,a) &= \int_0^a y\exp\left\{-\frac{y^2}{2t}\right\}+ y \sum_{n=1}^{\infty}\exp\left\{\frac{-n^2\pi^2}{4\alpha_t}\right\}\ dy
\qquad \mbox{ and }\\
I_2(t,a) &= \int_0^a 4\pi^{-1/2}\alpha_t^{3/2} \sum_{n=1}^\infty n^2\exp\{-\alpha_t n^2\} \times \min\{ y,a-y\}\ dy\ .
\end{align*}
Standard calculations reveal that
\begin{align*}
    I_1(t,a) &= \frac{a^2}{2}\left[ \frac{4\alpha_t}{\pi^2}\left(1- \exp\left\{-\frac{\pi^2}{4\alpha_t}\right\} \right) + \sum_{n=1}^{\infty}\exp\left\{\frac{-n^2\pi^2}{4\alpha_t}\right\}\right]\qquad \mbox{ and }\\
    I_2(t,a) &= \frac{a^2}{2} \left[ 2\pi^{-1/2}\alpha_t^{3/2} \sum_{n=1}^\infty n^2\exp\{-\alpha_t n^2\}\right]
\end{align*}
and thus $\alpha_t$ is the sole factor in determining which integral is larger. Figure \ref{fig:int:comparison} displays a graph of the two integrals (scaled by $2/a^2$), as functions of $\alpha_t$, evaluated numerically. 

\begin{figure}[tbp]
    \centering
    \includegraphics[width=0.85\linewidth]{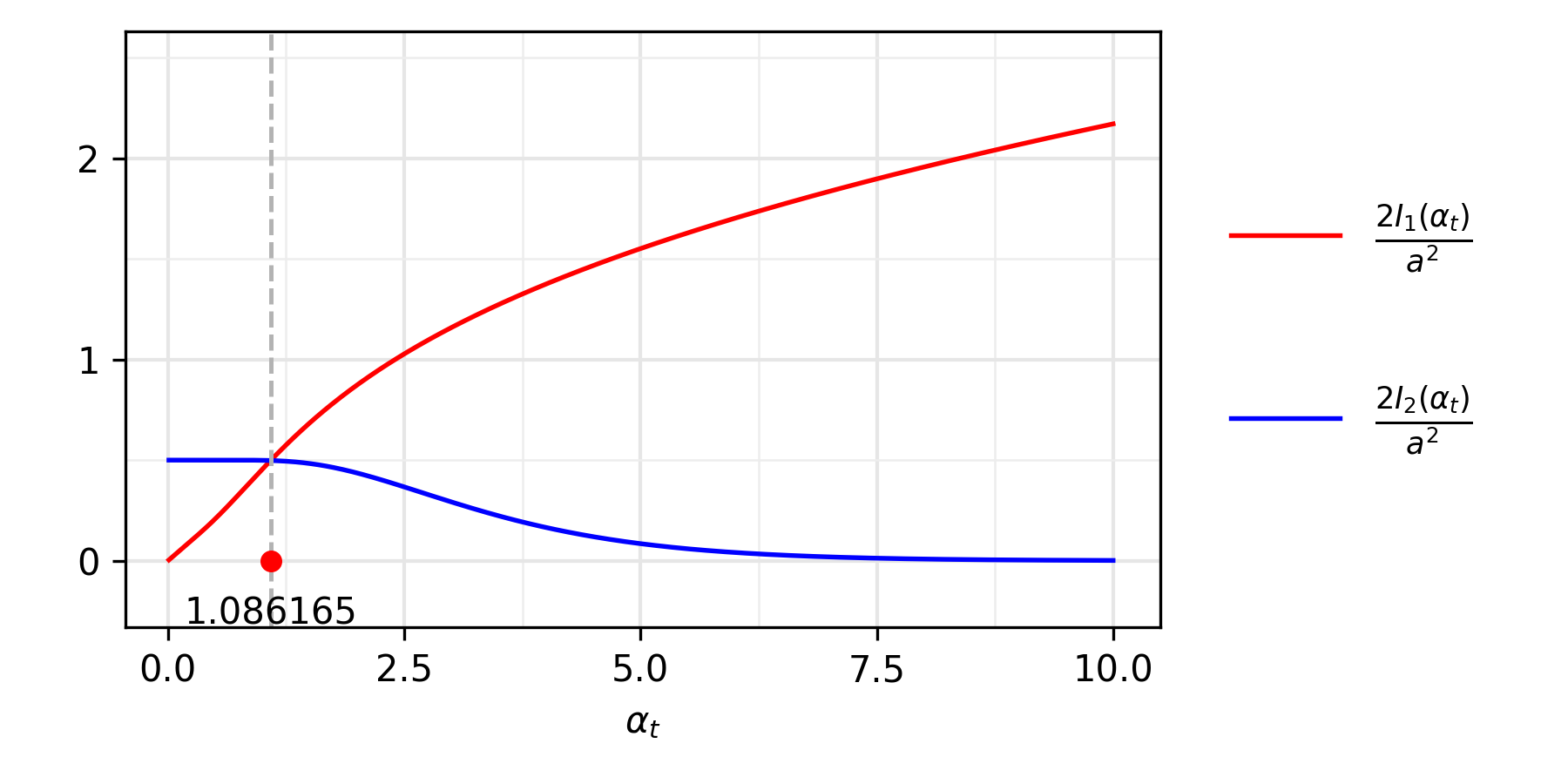}
    \caption{Integrals $I_1(t,a)$ (red) and $I_2(t,a)$ (blue), scaled by the factor $(2/a^2)$, displayed as functions of $\alpha_t$.}
    \label{fig:int:comparison}
\end{figure}

Numerically solving $I_1(t,a)=I_2(t,a)$ gives the crossover $\alpha_0\approx1.086165$ and the corresponding value $\beta_0=\pi^2/(4\alpha_0)\approx2.271664$, both rounded to six decimal places. We note that the exact values of $\alpha_0$ or $\beta_0$ are not necessary, since the bounds in \eqref{eq:bound:h1} and \eqref{eq:bound:h2} hold for any $a,t$. 

We can now design an efficient rejection sampler for $h$. For $\alpha_t<\alpha_0$ the bound \eqref{eq:bound:h1} will be sharper, while for $\alpha_t\ge \alpha_0$, the bound \eqref{eq:bound:h2} is better. 

\textit{Case 1: $\alpha_t<\alpha_0$ $(\beta_t>\beta_0)$}. We recognize the first term in the bound \eqref{eq:bound:h1} as being proportional to a Rayleigh random variable 
$$
\sR \eqd \sqrt{ -2t \log\left( 1 - \sU \cdot \left(1 - \exp\left\{-\beta_t \right\} \right) \right) },\qquad
f_\sR(y)\propto y\exp\left\{-\frac{y^2}{2t}\right\}\bfone\{0\le y\le a\}\ ,
$$ 
while the second term as being proportional to the density of $a\sqrt{\sU}$. Thus, defining $N_0 = N_0(\alpha_t,\beta_t)$ as in \eqref{eq:N0}, and the constants
\begin{equation}
\begin{aligned}
\tilde D_1 &= t(1-\exp\{-a^2/(2t)\}),\\
\tilde D_2 &= \frac{a^2}{2}
\left(
\sum_{n\ge 1}^{N_0+1} \exp\left\{-\beta_t n^2\right\}
+ \exp\left\{-\beta_t (N_0+1)^2\right\}
\right)
\ ,
\end{aligned}
\label{eq:D1:D2}
\end{equation}
then \eqref{eq:bound:h1} becomes
\begin{align}
    \label{eq:bound:h1:IS}
    h(y;t,a) &\le \tilde D_1 f_{\sR}(y) + \tilde D_2 f_{a\sqrt{\sU}}(y)\ ,\quad 0\le y\le a. \nonumber \\
    & = (\tilde D_1 + \tilde D_2) \left(\frac{\tilde D_1}{\tilde D_1 + \tilde D_2}f_{\sR}(y) + \frac{\tilde D_2}{\tilde D_1 + \tilde D_2}f_{a\sqrt{\sU}}(y)\right) \ ,\quad 0\le y\le a. \nonumber \\ 
    &\stackrel{\rm def}{=} (\tilde D_1 + \tilde D_2) q_h^{(1)}(y)\ ,
\end{align}
where the definition of the mixture density $q_h^{(1)}$ is clear from above. The procedure for sampling from $h$ in this case is then given in Algorithm \ref{alg:draw:h1}.
\begin{algorithm}[!ht]
  \caption{Exact Simulation of $W_t^0\mid A_0(t)$, when $\alpha<\alpha_0$.}
  \label{alg:draw:h1}
  \begin{algorithmic}[1]
    \STATE Input $t,a$; Calculate $\alpha_t, \beta_t, \tilde D_1, \tilde D_2$.
    \STATE Calculate the mixture weight $p = \tilde D_1/(\tilde D_1 + \tilde D_2)$;
    \STATE Simulate $\sU, \sU_1, \sU_2$;
    \STATE If $\sU_1<p$ set $Y = \sqrt{ -2t \log\left( 1 - \sU_2 \cdot \left(1 - \exp\left\{-\beta_t \right\} \right) \right) } $
    ; Else, set $Y = a\sqrt{\sU_2}$;
    \STATE If $\sU\cdot (\tilde D_1 + \tilde D_2) \cdot q_h^{(1)}(Y;t,a) \lerh h(Y;t,a)$ accept. Else, reject and return to STEP 3. 
    \STATE Output $Y$.
  \end{algorithmic}
\end{algorithm}
Recalling that the normalizing constant for $h$ is $t\PP(\sup_{0\le s\le t}W^{\rm me}_s \le a)$, see \eqref{eq:h0}, the acceptance rate of this algorithm is 
$$
\tilde A_1 = \frac{t\PP\left(\sup_{0\le s\le t}W^{\rm me}_s\le a\right)}{\tilde D_1 + \tilde D_2}\ .
$$
It is well known that $\sup_{0\le s\le t}W^{\rm me}_s \eqd 2\sqrt{t}\sK$, where $\sK$ has a Kolmogorov-Smirnov distribution \cite{Kolmogorov1933}. Then, 
\begin{align*}
\PP\left(\sup_{0\le s\le t}W^{\rm me}_s \le a\right) &= \PP\left(\sK \le \frac{a}{2\sqrt{t}}\right) 
 = 1 + 2\sum_{n=1}^\infty (-1)^n \exp\left\{-n^2\frac{a^2}{2t}\right\}\\
 &= 1 + 2\sum_{n=1}^\infty (-1)^n \exp\left\{-\beta_t n^2\right\}\ .
\end{align*}
Algorithm \ref{alg:draw:h1} requires a computable value $\tilde D_2$, rather than the exact series representation, in order to evaluate the mixture weight $p$. Based on the envelope in \eqref{eq:bound:h1} the hypothetical acceptance probability of Algorithm \ref{alg:draw:h1}, can be defined as
$$
A_1 = A_1(\beta_t) = \frac{\displaystyle 1 + 2\sum_{n=1}^\infty (-1)^n \exp\left\{- \beta_t n^2\right\}}
{\displaystyle \left(1-\exp\left\{-\beta_t \right\}\right) + \beta_t \sum_{n=1}^\infty \exp\left\{-\beta_t n^2\right\} }
$$
However, the denominator is replaced by the computable upper bound in \eqref{eq:D1:D2}, and thus the actual acceptance probability becomes 
\[
\widetilde A_1 = \tilde{A_1}(\beta_t)
=
A_1(\beta_t)
\frac{ t\left[\displaystyle \left(1-\exp\left\{-\beta_t \right\}\right) + \beta_t \sum_{n=1}^\infty \exp\left\{-\beta_t n^2\right\} \right]}{\widetilde D_1+\widetilde D_2}
\le A_1(\beta_t).
\]
The next result establishes an interesting property of $A_1$, which can then be used to construct a lower bound for $\tilde{A_1}$.

\begin{lemma}
    On the domain $\beta_t \in (\beta_0, \infty)$, the function $A_1(\beta_t)$ is increasing in $\beta_t$.
    \label{lemma:acc:I}
\end{lemma}
\begin{proof}
    See Appendix \ref{proof:lemma:acc:I}.
\end{proof}

We now quantify the difference between $A_1(\beta_t)$ and the actual
acceptance probability $\widetilde A_1$. Fix $\beta_t\ge\beta_0$ and
recall that $a^2/2=t\beta_t$. Integrating the envelope in \eqref{eq:bound:h1} gives
\[
I_1(t,a)
=
t\left[
1-e^{-\beta_t}
+\beta_t\sum_{n=1}^{\infty}e^{-\beta_t n^2}
\right].
\]
This gives
\[
\widetilde D_1+\widetilde D_2-I_1(t,a)
=
t\beta_t\left[
e^{-\beta_t(N_0+1)^2}
-\sum_{n=N_0+2}^{\infty}e^{-\beta_t n^2}
\right]
\le 
t\beta_t e^{-\beta_t(N_0+1)^2},
\]
where the last inequality follows from Proposition~\ref{prop:h1}. Moreover, $\beta_t\ge\beta_0$ implies $N_0\ge2$. We therefore obtain the uniform bound
\[
0\le
\widetilde D_1+\widetilde D_2-I_1(t,a)
\le t\beta_t e^{-9\beta_t}.
\]

We also need a lower bound on $I_1(t,a)$. Retaining only the first
term of its positive series, and using $\beta_t\ge\beta_0>1$, gives
\[
\frac{I_1(t,a)}{t}
\ge
1-e^{-\beta_t}+\beta_t e^{-\beta_t}
=
1+(\beta_t-1)e^{-\beta_t}
\ge1.
\]
Combining the last two bounds, we have
\[
0\le
\frac{\widetilde D_1+\widetilde D_2-I_1(t,a)}{I_1(t,a)}
\le \beta_t e^{-9\beta_t}.
\]
Finally, note that
\[
\widetilde A_1
=
\frac{A_1(\beta_t)}
{1+
\displaystyle
\frac{\widetilde D_1+\widetilde D_2-I_1(t,a)}{I_1(t,a)}
},
\]
and therefore
\[
\frac{A_1(\beta_t)}{1+\beta_t e^{-9\beta_t}}
\le \widetilde A_1
\le A_1(\beta_t).
\]
By Lemma~\ref{lemma:acc:I}, $A_1(\beta_t)\ge A_1(\beta_0)$, and using the fact that  
$\beta e^{-9\beta}$ is decreasing for $\beta\ge\beta_0$, it follows that
\[
\widetilde A_1
\ge
\frac{A_1(\beta_t)}{1+\beta_t e^{-9\beta_t}}
\ge
\frac{A_1(\beta_0)}{1+\beta_0 e^{-9\beta_0}} > 0.70.
\]
Thus, Algorithm 1 has an acceptance probability exceeding $70\%$
throughout its stated regime. Finally, since
$A_1(\beta_t)\to1$ and $\beta_t e^{-9\beta_t}\to0$ as
$\beta_t\to\infty$, the preceding two-sided bound also gives
$\widetilde A_1\to1$ as $\beta_t\to \infty$.

\textit{Case 2: $\alpha_t\ge \alpha_0$ ($\beta_t \le \beta_0$).} In this case the bound \eqref{eq:bound:h2} is better and we recognize the last term in \eqref{eq:bound:h2} as being proportional to the density of the random variable
\begin{equation}
\ell(\sU) = 
\begin{cases}
 (a/\sqrt{2}) \sqrt{\sU}, &  \sU < 0.5\ , \\
 a - (a/\sqrt{2}) \sqrt{1 - \sU}, &  \sU \ge 0.5\ .
\end{cases}
\qquad f_{\ell(\sU)}(y) \propto \min\{y,a-y\}\bfone\{0\le y\le a\}\ .
\label{eq:q_h}
\end{equation}
Defining the constant
$$
\tilde D_3 = 4\pi^{-1/2}\alpha_t^{3/2} \sum_{n=1}^\infty n^2\exp\{-\alpha_t n^2\}\cdot \frac{a^2}{4}\ ,
$$
we see that \eqref{eq:bound:h2} gives that 
$$
h(y;t,a)\le \tilde D_3 f_{\ell(\sU)}(y),\ \  0\le y\le a\ .
$$
The rejection sampling procedure for simulating from $h$ in this case is given in Algorithm \ref{alg:draw:h2}.

\begin{algorithm}[tbp]
  \caption{Exact Simulation of $W_t^0\mid A_0(t)$ when $\alpha\ge \alpha_0$.}
  \label{alg:draw:h2}
  \begin{algorithmic}[1]
    \STATE Input $t,a$; 
    \STATE Simulate $\sU, \sU_1$; Set $Y = \ell(\sU_1)$;
    \STATE If $\sU \cdot \tilde D_3 \cdot f_{\ell(\sU_1)}(Y) \lerh h(Y;t,a)$ accept. Else, return to STEP 2. 
    \STATE Output $Y$.
  \end{algorithmic}
\end{algorithm}

We analyze the acceptance rate for Algorithm \ref{alg:draw:h2} similarly, as follows: 
\begin{align*}
A_2 = A_2(\beta_t) &= \frac{t\PP\left(\sup_{0\le s\le t}W^{\rm me}_s \le a\right)}{\tilde D_3} \\
&= \frac{\displaystyle 1 + 2\sum_{n=1}^\infty (-1)^n \exp\left\{- \beta_t n^2\right\}}
{\displaystyle 4\pi^{-1/2} \left(\frac{\pi^2}{4\beta_t}\right)^{3/2} \sum_{n=1}^\infty n^2\exp\left\{-\frac{n^2\pi^2}{4\beta_t}\right\}\cdot \frac{\beta_t}{2}} \\
&= \frac{\displaystyle 1 + 2\sum_{n=1}^\infty (-1)^n \exp\left\{- \beta_t n^2\right\}}
{\displaystyle 4^{-1}\pi^{5/2}\beta_t^{-1/2}\sum_{n=1}^\infty n^2\exp\left\{-\frac{n^2\pi^2}{4\beta_t}\right\}} \ .
\end{align*}
\begin{theorem}
    Over the domain $0<\beta_t\le \beta_0$, the function $A_2(\beta_t)$ is decreasing in $\beta_t$.
    \label{thm:acc:II}
\end{theorem}
\begin{proof}
    See Appendix \ref{proof:thm:acc:II}.
\end{proof}
As a consequence, for all $\beta_t\le\beta_0$, we have $A_2(\beta_t) \ge A_2(\beta_0)\approx 0.70$. Thus, Algorithm \ref{alg:draw:h2} has a minimum acceptance rate of $70\%$. In the limit, 
$$
\lim_{\beta_t\to 0}A_2(\beta_t) = \lim_{\alpha_t\to \infty}A_2(\pi^2/(4\alpha_t)) = \frac{8}{\pi^2}\ ,
$$
where the last equality follows easily from the expression of $A_2$ given in the proof of Theorem \ref{thm:acc:II}. Figure \ref{fig:acc:rate:h} illustrates $A_1(\beta_t)$ (blue curve) and $A_2(\beta_t)$ (red curve).

\begin{figure}[tbp]
    \centering
    \includegraphics[width=0.85\linewidth]{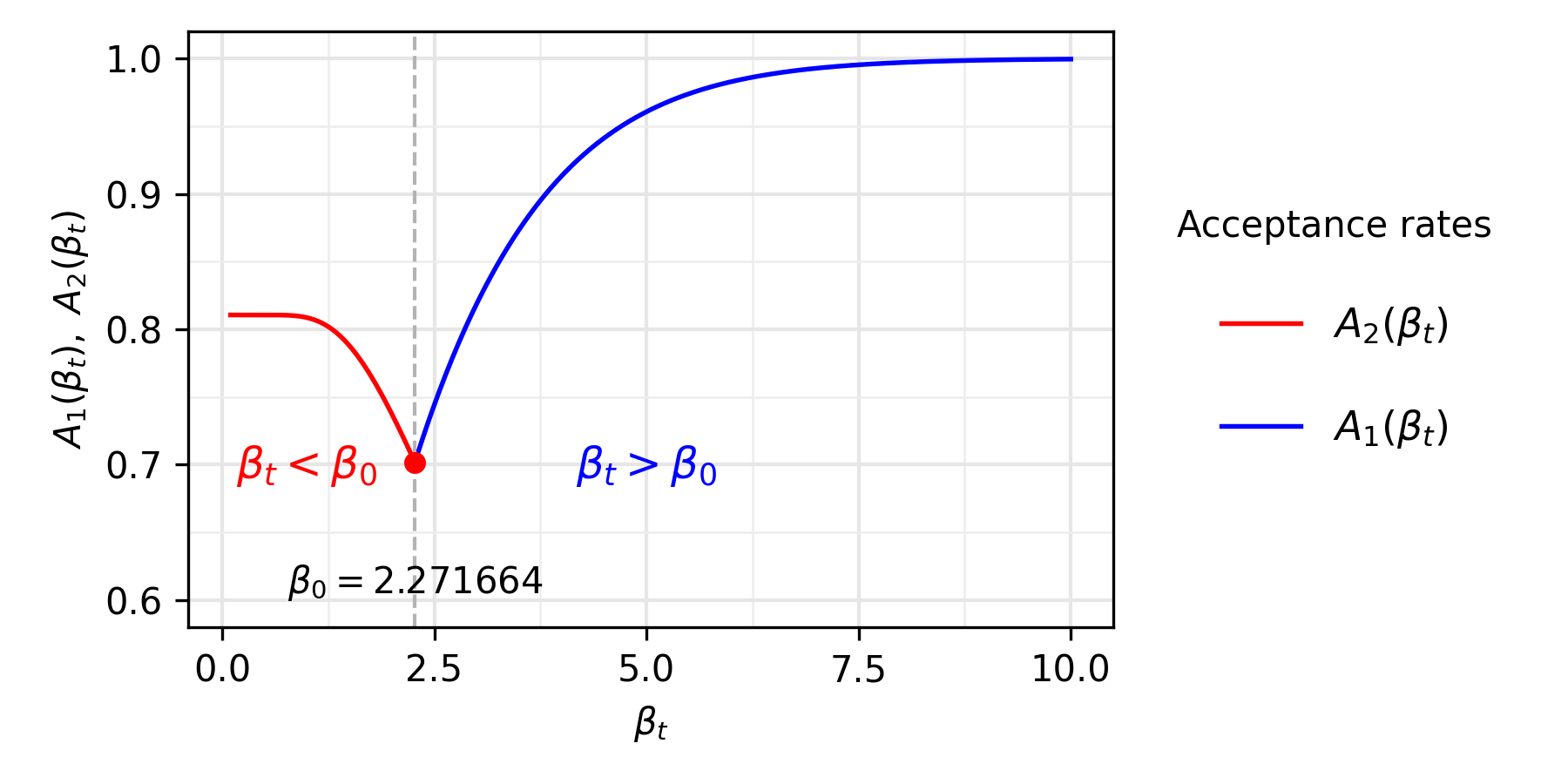}
    \caption{The acceptance rate of Algorithm \ref{alg:draw:h1} (blue line) and Algorithm \ref{alg:draw:h2} (red line) as a function of $\beta_t$.}
    \label{fig:acc:rate:h}
\end{figure}

\subsubsection{Approximating sequences}
Both Algorithms \ref{alg:draw:h1} and \ref{alg:draw:h2} require the use of a \emph{series comparison} at the accept/reject step. We now specify the necessary approximating sequences, so that this comparison can be performed exactly, without the need to approximate series with finite sums.  We note that the constant $\tilde D_1$ and $\tilde D_2$ can be evaluated exactly, thus, there is no need for a further approximation. The constant $\tilde D_3$ can be approximated as follows.
Based on \eqref{eq:bound S1}, for $N >  N_0(\alpha_t, \beta_t)$ defined in \eqref{eq:N0}, we have that
$$
\pi^{-1/2}\alpha_t^{3/2} a^2 S^{(1)}_{N+1}(\alpha_t) 
<  \tilde D_3 
<  \pi^{-1/2}\alpha_t^{3/2} a^2\Bigl[S^{(1)}_{N+1}(\alpha_t) + 2(N+1)^2\exp\{-\alpha_t(N+1)^2\}\Bigr]\ .
$$
Next, consider the sequence of partial sums and the remainder term as derived from the $\sin$ series defining both $g$ and $h$, see equations \eqref{eq:g2} and \eqref{eq:h2}: 
\begin{align*}
S^g_N(y) &\stackrel{\rm def}{=} \frac{2}{a} \sum_{n=1}^N  \sin\left( \frac{\pi n x}{a} \right)\sin\left( \frac{\pi n y}{a} \right) \exp\left\{ -\frac{\pi^2  n^2}{2a^2} t \right\},\\
R_N^g(y) &\stackrel{\rm def}{=} g(y;t,x,a) - S^g_N(y), \\ 
S^h_N(y) &\stackrel{\rm def}{=} K^h \sum_{n=1}^N n  \sin\left( \frac{\pi n y}{a} \right) \exp\left\{ -\frac{\pi^2  n^2}{2a^2} t \right\},\\
R_N^h(y) &\stackrel{\rm def}{=} h(y;t,a) - S^h_N(y),
\end{align*}
where $N\ge 1$. Then, using Proposition \ref{prop:h1} and the fact that $|\sin(x)| \le |x|$, we deduce quickly that
\begin{align*}
|R^g_N(y)| &\le \frac{2xy\pi^2}{a^3}\sum_{n=N+1}^\infty n^2\exp\{-\alpha_t n^2\} \le \frac{2xy\pi^2}{a^3} x^{(1)}_N(\alpha_t) \stackrel{\rm def}{=} x^g_N(y)\ ,\\
|R^h_N(y)| &\le K^h\frac{\pi y}{a}\sum_{n=N+1}^\infty n^2\exp\{-\alpha_t n^2\} \le K^h\frac{\pi y}{a} x^{(1)}_N(\alpha_t) \stackrel{\rm def}{=} x^h_N(y)\ .
\end{align*}
The next two results provide the remaining necessary ingredients.
\begin{proposition}
\label{prop:g2}
With $S^g_N$, $R^g_N$, and $x^g_N$ defined above, when $N > N_0(\alpha_t,\beta_t)$ and $y\in [0,a]$, we have that
\begin{enumerate}
    \item $S^g_N(y) + x^g_N(y) \ge g(y;t,x,a) \ge S^g_N(y)-x^g_N(y)$;
    \item $S^g_N(y) -x^g_N(y) \nearrow g(y;t,x,a)$ as $N\to \infty$;
    \item $S^g_N(y) + x^g_N(y) \searrow g(y;t,x,a)$ as $N\to \infty$.
\end{enumerate}
\end{proposition}
\begin{proposition}
\label{prop:h2}
With $S^h_N$, $R^h_N$, and $x^h_N$ defined above, when $N > N_0(\alpha_t,\beta_t)$ and $y\in [0,a]$, we have that
\begin{enumerate}
    \item $S^h_N(y) + x^h_N(y) \ge h(y;t,a) \ge S^h_N(y)-x^h_N(y)$;
    \item $S^h_N(y) -x^h_N(y) \nearrow h(y;t,a)$ as $N\to \infty$;
    \item $S^h_N(y) + x^h_N(y) \searrow h(y;t,a)$ as $N\to \infty$.
\end{enumerate}
\end{proposition}

\begin{proof}
    See Appendix \ref{proof:h2}.
\end{proof}

The image representations provide alternative approximating sequences when
$\alpha_t$ is small. For integers $N\geq 0$, define
\begin{align}
\widetilde S_N^g(y)
&=\frac{1}{\sqrt{2\pi t}}
  \sum_{n=-N}^{N}
  \left[
  \exp\left\{-\frac{(y+2na-x)^2}{2t}\right\}
  -\exp\left\{-\frac{(y+2na+x)^2}{2t}\right\}
  \right],
\label{eq:image-partial-g}\\
\widetilde R_N^g(y)
&=g(y;t,x,a)-\widetilde S_N^g(y),
\notag\\
\widetilde S_N^h(y)
&=\sum_{n=-N}^{N}(2na+y)
  \exp\left\{-\frac{(2na+y)^2}{2t}\right\},
\label{eq:image-partial-h}\\
\widetilde R_N^h(y)
&=h(y;t,a)-\widetilde S_N^h(y).
\notag
\end{align}
Set
\begin{align}
\widetilde x_N^g(y)
&=\frac{
  \displaystyle
  \exp\left\{-\frac{[2(N+1)a-x-y]^2}{2t}\right\}}
  {\displaystyle
  \sqrt{2\pi t}\left[
  1-\exp\left\{-\frac{2a[(2N+3)a-x-y]}{t}\right\}
  \right]},
\label{eq:image-error-g}\\
\widetilde x_N^h(y)
&=[2(N+1)a-y]
  \exp\left\{-\frac{[2(N+1)a-y]^2}{2t}\right\},
\label{eq:image-error-h}
\end{align}
and define
\begin{equation}
\widetilde N_0(\beta_t)
=\max\left\{
0,\left\lceil
\frac{(2\beta_t)^{-1/2}-1}{2}
\right\rceil
\right\},
\qquad \beta_t=\frac{a^2}{2t}.
\label{eq:image-cutoff}
\end{equation}
The denominator in \eqref{eq:image-error-g} is strictly positive for
$0<x<a$, $y\in[0,a]$, and $N\geq0$.

\begin{proposition}\label{prop:image-g}
With $\widetilde S_N^g$, $\widetilde R_N^g$, and
$\widetilde x_N^g$ defined above, for $a,t>0$, $0<x<a$,
$y\in[0,a]$, and every integer $N\geq0$, we have that
\begin{enumerate}
\item
$\widetilde S_N^g(y)+\widetilde x_N^g(y)
\geq g(y;t,x,a)
\geq\widetilde S_N^g(y)-\widetilde x_N^g(y)$;
\item
$\widetilde S_N^g(y)-\widetilde x_N^g(y)
\nearrow g(y;t,x,a)$ as $N\to\infty$;
\item
$\widetilde S_N^g(y)+\widetilde x_N^g(y)
\searrow g(y;t,x,a)$ as $N\to\infty$.
\end{enumerate}
\end{proposition}

\begin{proposition}\label{prop:image-h}
With $\widetilde S_N^h$, $\widetilde R_N^h$, and
$\widetilde x_N^h$ defined above, for $a,t>0$,
$y\in[0,a]$, and every integer
$N\geq\widetilde N_0(\beta_t)$, we have that
\begin{enumerate}
\item
$\widetilde S_N^h(y)+\widetilde x_N^h(y)
\geq h(y;t,a)
\geq\widetilde S_N^h(y)-\widetilde x_N^h(y)$;
\item
$\widetilde S_N^h(y)-\widetilde x_N^h(y)
\nearrow h(y;t,a)$ as $N\to\infty$;
\item
$\widetilde S_N^h(y)+\widetilde x_N^h(y)
\searrow h(y;t,a)$ as $N\to\infty$.
\end{enumerate}
In fact, the sharper upper bound
$h(y;t,a)\leq\widetilde S_N^h(y)$ holds, and
$\widetilde S_N^h(y)\searrow h(y;t,a)$.
\end{proposition}

\begin{proof}
See Appendix~\ref{app:image-bounds}.
\end{proof}

In particular, $\widetilde N_0(\beta_t)=0$ whenever
$\beta_t\geq1/2$. The remainder bounds decrease rapidly with $N$ when
$\alpha_t$ is small, since $\beta_t=\pi^2/(4\alpha_t)$ and
\begin{align}
\widetilde x_N^g(y)
&\leq
\frac{\exp\{-4\beta_tN^2\}}
{\sqrt{2\pi t}\,[1-\exp\{-4\beta_t(2N+1)\}]},
\label{eq:image-uniform-g}\\
\widetilde x_N^h(y)
&\leq 2(N+1)a\exp\{-\beta_t(2N+1)^2\}.
\label{eq:image-uniform-h}
\end{align}
These inequalities hold uniformly in $y\in[0,a]$; the first also
holds uniformly in $x\in(0,a)$.

The spectral and image representations are both valid for every
$\alpha_t>0$, subject to their respective starting-index conditions.
Their relative efficiency can be assessed by comparing the remainder
bounds. For fixed $a,t>0$ and interior spatial arguments, as $N\to\infty$, and $k\in \{g,h\}$, we have
\begin{align*}
\log x_N^k(y)
&=-\alpha_tN^2+O(N+\log N),
\\
\log\widetilde x_N^k(y)
&=-\frac{\pi^2}{\alpha_t}N^2+O(N+\log N).
\end{align*}
The second relation follows from the spacing $2a$ between consecutive
images and the identity $4\beta_t=\pi^2/\alpha_t$. Thus, at a common
truncation index, the image bounds have faster asymptotic decay when
$\alpha_t<\pi$, whereas the spectral bounds have faster asymptotic
decay when $\alpha_t>\pi$. A natural switching rule is therefore to
use the image representation for $\alpha_t<\pi$ and the spectral
representation for $\alpha_t\geq\pi$.

This cutoff provides an asymptotic guide to the choice of representation.
The most economical choice for a particular evaluation also depends on
the spatial arguments, the coefficients in the remainder bounds, and
the cost of evaluating an additional image pair or spectral term.
Either representation yields the same exact comparison through its monotone bounds, so the switching rule affects computational efficiency without affecting the validity of the sampling procedure.

Our simulation results are presented in Figure \ref{fig:draws:h}, where we display histograms of $10,000$ draws from $h$ based on Algorithms \ref{alg:draw:h1} and \ref{alg:draw:h2}, kernel density estimates (blue curves) and theoretical densities (red curves) for two cases: $a=2, t=0.2$ (left) and $a=2, t=1$ (right).

\begin{figure}[tbp]
    \centering
    \includegraphics[width=\linewidth]{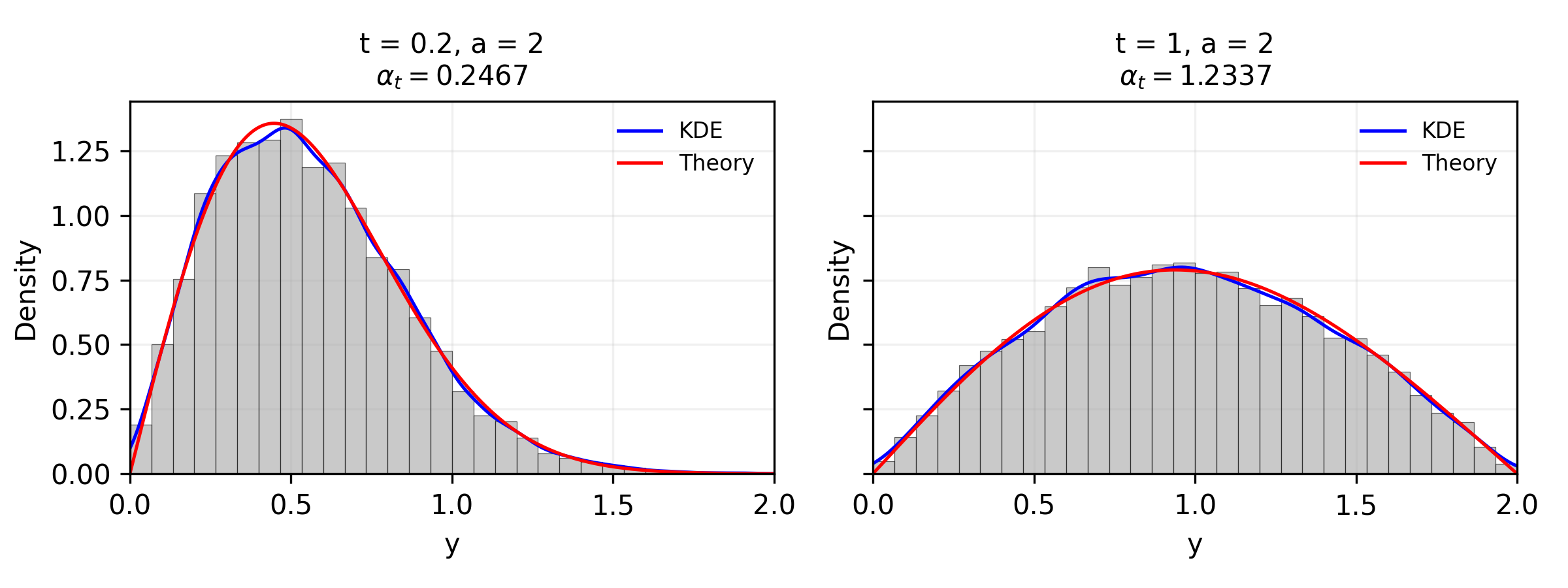}
    \caption{Histograms (grey) of $10,000$ draws from the density proportional to $h(\cdot;t,a)$, with kernel density estimates (blue curves) and numerically normalized theoretical densities (red curves). In the left panel we set $t=0.2, a=2$, while in the right panel we set $t=1, a=2$.}
    \label{fig:draws:h}
\end{figure}

\subsection{Interpolation under survival}
\label{sec:interp}

We now consider our second goal, that is, to develop an exact sampling algorithm for the variate $W_t^{x\to z}$, conditionally on the event $B$. For simplicity, let $Y_t^{x\to z} \eqd W_t^{x\to z}\mid B$. In short, we wish to sample a Brownian bridge from $x\in [0,a]$ to $z\in [0,a]$, at time $t\in (0,T)$, conditionally that it never leaves the interval $[0,a]$. We distinguish three cases: (i) $0<x,z<a$, (ii) $0=x<z<a$ and (iii) $0=x<z=a$. The remaining scenarios can be reduced to one of these three situations via symmetry, or are straight forward extensions of the tools and techniques developed below.

\textbf{Case I}: $0<x,z<a$. Using Proposition \ref{prop:markov}, the un-normalized density of interest is then, 
\begin{align*}
\PP(Y_t^{x\to z}\in dy) &\propto k^*_{(1)}(y)\,dy
\stackrel{\rm def}{=} g(y;t,x,a)\times g(y;T-t,z,a)\,dy, \quad y\in (0,a)\ .
\end{align*}
Using the constants $C_g^{(1)}$ and $C_g^{(2)}$ -- provided in Proposition \ref{prop:bound:g}, we consider the following two cases.
\begin{enumerate}
\item If $a/\sqrt{T-t} < 2$, we use the bound provided in \eqref{eq:bound:g:1} and write
$$
k^*_{(1)}(y) \le g(y;t,x,a) C^{(1)}_g(T-t,z,a)\ .
$$
Thus, a rejection sampling algorithm  will propose $Y\sim g(y;t,x,a)$ and accept this proposal as a draw from $k^*_{(1)}$ with probability
$$
\frac{k^*_{(1)}(Y)}{g(Y,t,x,a))C^{(1)}_g(T-t,z,a)} = \frac{g(Y;T-t,z,a)}{C^{(1)}_g(T-t,z,a)}\ .
$$

\item If $a/\sqrt{T-t}\ge 2$, we use the bound provided in \eqref{eq:bound:g:2}. It follows that
$$
k^*_{(1)}(y)\le g(y;t,x,a)C_g^{(2)}(T-t,z,a)
$$
and a proposed value $Y\sim g(y;t,x,a)$ is accepted as a draw from $k^*_{(1)}$ with probability
$$
\frac{k^*_{(1)}(Y)}{g(Y;t,x,a)C^{(2)}(T-t,z,a)} = \frac{g(Y;T-t,z,a)}{C_g^{(2)}(T-t,z,a)}\ .
$$
\end{enumerate}
We combine these two cases and present the complete sampling procedure in Algorithm \ref{alg:interp:I}. This technique can be extended to simulate any exact discrete path $(Y_{t_1}^{x\to z}, Y_{t_2}^{x\to z}, \ldots, Y_{t_n}^{x\to z})$, where $0=t_0<t_1<\ldots<t_n< t_{n+1}=T$ are given. Using the Markov property, this can be achieved sequentially. Set $Y_{t_0}^{x\to z}=x$, $Y_{t_{n+1}}^{x\to z} = z$. For $i=1,\ldots, n$, simulate $Y_{t_i}^{x\to z}$ using Algorithm \ref{alg:interp:I} with input parameters
$$
\Big(\ t_{i} - t_{i-1},\, T-t_{i-1},\, Y_{t_{i-1}}^{x\to z},\, Y_{t_{n+1}}^{x\to z},\, a\Big)\ .
$$

\begin{algorithm}[tbp]
  \caption{Exact simulation of $Y_t^{x\to z}$, $0<x,z<a$ }
  \label{alg:interp:I}
  \begin{algorithmic}[1]
    \STATE Input $t,T,x,z,a$; 
    \STATE If $a/\sqrt{T-t}< 2$ set $C = C^{(1)}_g(T-t,z,a)$. Else, set $C = C^{(2)}_g(T-t,z,a)$;
    \STATE Simulate $Y\sim g(y;t,x,a)$; Simulate $\sU$;
    \STATE If $\sU\times C\lerh g(Y;T-t,z,a)$, accept. Else, reject and return to STEP 3. 
    \STATE Output $Y$.
  \end{algorithmic}
\end{algorithm}

Our simulation results are presented in Figure \ref{fig:interp:I}. The upper two panels present histograms (grey), kernel density estimates (blue) and theoretical densities (red) for the marginal distributions of $Y_{0.2}^{0.1\to 1.8}$ and $Y_{2.8}^{0.1\to 1.8}$, where we set $a=2$ and $T=3$. These histograms are based on $10,000$ joint draws of $\big(Y_{0.2}^{0.1\to 1.8}, Y_{2.8}^{0.1\to 1.8}\big)$, generated as described above. The lower panel shows five separately simulated constrained Brownian bridge paths (various colors) from $x=0.1$ to $z=1.8$, with colored dots at $t=0.2$ and $t=2.8$ and black dots at the fixed endpoints. The sideways histograms and kernel density estimates use the same joint draws as the upper panels, with theoretical densities (red) superimposed.

\begin{figure}[tbp]
    \centering
    \includegraphics[width=0.95\linewidth]{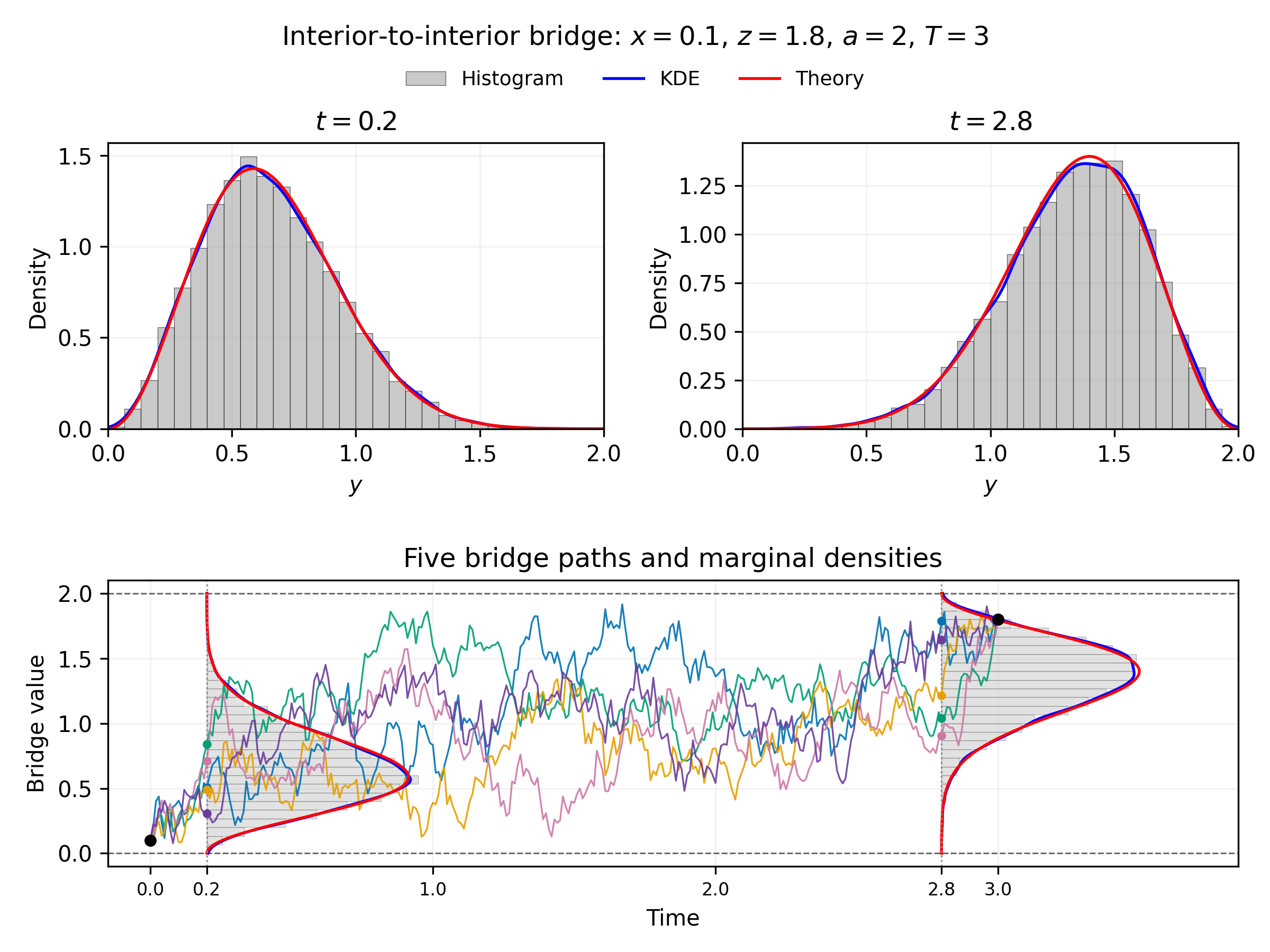}
    \caption{An illustration of Algorithm \ref{alg:interp:I}. The global parameters are set to $a=2, T=3, x=0.1, z=1.8$. Upper panels: marginal histograms (grey) based on $10,000$ joint draws of $(Y_{0.2}^{x\to z}, Y_{2.8}^{x\to z})$, kernel density estimates (blue) and theoretical densities (red), with $t=0.2$ on the left and $t=2.8$ on the right. The lower panel shows five separately simulated constrained Brownian bridge paths (various colors), with black dots marking the common start and end points, $(0,x)$ and $(T,z)$. The sideways marginal histograms (grey) use the same $10,000$ joint draws as the upper panels, with kernel density estimates (blue) and theoretical densities (red) superimposed.}
    \label{fig:interp:I}
\end{figure}

\noindent
\textbf{Case II}: $0=x < z < a$. Repeating the argument from Section \ref{sec:extrap} (Case II, $x=0$), we find that the density of $Y_t^{0\to z}$ is then proportional to
\begin{align}
    \PP(Y_t^{0\to z}\in dy) \propto k^*_{(2)}(y)\,dy
    \stackrel{\rm def}{=} h(y;t,a)\times g(y;T-t,z,a)\,dy\quad y\in (0,a).
\end{align}
Using a similar argument as above, a draw $Y\sim h(y;t,a)$ is accepted with probability
$$
\frac{k^*_{(2)}(Y)}{h(Y;t,a)\times C} 
= \frac{g(Y;T-t,z,a)}{C}\ ,
$$
where we set $C = C_g^{(1)}(T-t,z,a)$ if $a/\sqrt{T-t}<2$ and $C = C_g^{(2)}(T-t,z,a)$ if $a/\sqrt{T-t}\ge 2$. The complete procedure is given in Algorithm \ref{alg:interp:II}. For any set of time points $0=t_0 < t_1 < \ldots <t_n <t_{n+1} =T$, the exact discrete path $\big(Y_{t_{1}}^{0\to z}, Y_{t_2}^{0\to z}, \ldots, Y_{t_n}^{0\to z}\big)$ can be simulated jointly in the following sequential way. Set $Y_{t_0}^{0\to z} = 0$, $Y_{t_{n+1}}^{0\to z} = z$. Simulate $Y_{t_1}^{0\to z}$ using Algorithm \ref{alg:interp:II} with inputs $(t_1, T, z, a)$. For $i=2,\ldots, n$, simulate $Y_{t_i}^{0\to z}$ using Algorithm \ref{alg:interp:I} with inputs
$$
\left( t_i - t_{i-1},\ T - t_{i-1},\ Y_{t_{i-1}}^{0\to z},\  Y_{t_{n+1}}^{0\to z},\,a \right)\ .
$$

\begin{algorithm}[tbp]
  \caption{Exact Simulation of $Y_t^{0\to z}$.}
  \label{alg:interp:II}
  \begin{algorithmic}[1]
    \STATE Input $t,T,z,a$; 
    \STATE If $a/\sqrt{T-t}< 2$ set $C = C^{(1)}_g(T-t,z,a)$. Else, set $C = C^{(2)}_g(T-t,z,a)$;
    \STATE Simulate $Y\sim h(y;t,a)$; Simulate $\sU$;
    \STATE If $\sU\times C\lerh g(Y;T-t,z,a)$, accept. Else, reject and return to STEP 3. 
    \STATE Output $Y$.
  \end{algorithmic}
\end{algorithm}

We present the results of our simulations in Figure \ref{fig:interp:II}. We set $a=2, T=3, x=0, z=1.8$. The upper panels show marginal histograms (grey) based on $10,000$ joint draws of $(Y_{0.2}^{0\to z}, Y_{2.8}^{0\to z})$, kernel density estimates (blue) and theoretical densities (red), with $t=0.2$ on the left and $t=2.8$ on the right. The lower panel shows five separately simulated constrained Brownian bridge paths from $x=0$ to $z=1.8$, with colored dots at $t=0.2$ and $t=2.8$ and black dots at the fixed endpoints. The sideways histograms and kernel density estimates use the same joint draws as the upper panels, with theoretical densities (red) superimposed.

\begin{figure}[tbp]
    \centering
    \includegraphics[width=0.95\linewidth]{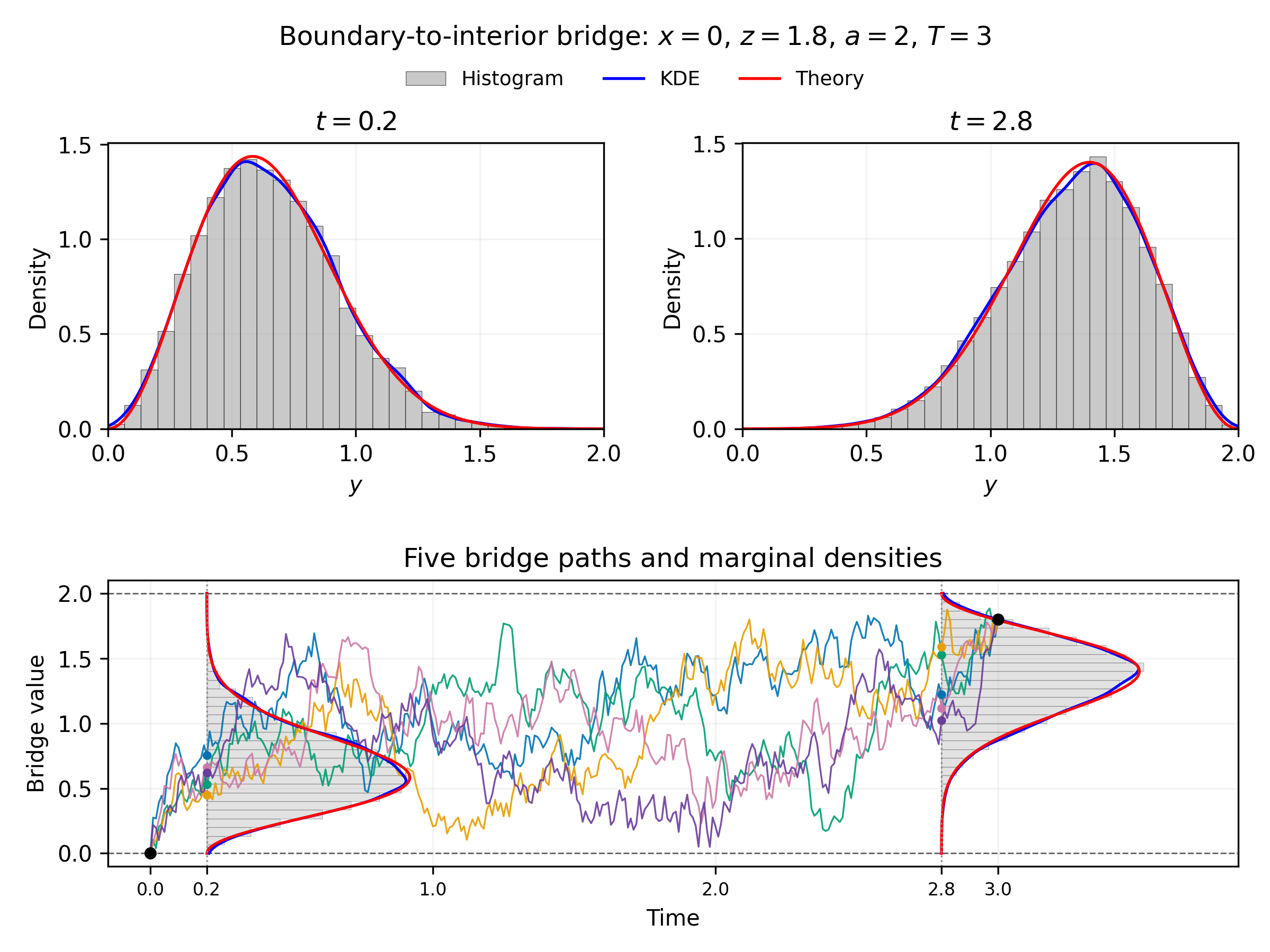}
    \caption{An illustration of Algorithm \ref{alg:interp:II}. The global parameters are set to $a=2, T=3, x=0, z=1.8$. Upper panels: marginal histograms (grey) based on $10,000$ joint draws of $(Y_{0.2}^{x\to z}, Y_{2.8}^{x\to z})$, kernel density estimates (blue) and theoretical densities (red), with $t=0.2$ on the left and $t=2.8$ on the right. The lower panel shows five separately simulated constrained Brownian bridge paths (various colors), with black dots marking the common start and end points, $(0,x)$ and $(T,z)$. The sideways marginal histograms (grey) use the same $10,000$ joint draws as the upper panels, with kernel density estimates (blue) and theoretical densities (red) superimposed.}
    \label{fig:interp:II}
\end{figure}

Before moving on the next case we note that the scenario when $0<x<z=a$ can be incorporated in Case II via a simple symmetry argument. Thus, if $Y_t^{x\to a} \eqd W_t^{x\to a}\mid B$ then, for $y\in (0,a)$ we have
$$
\PP(Y_t^{x\to a}\in dy) \propto k_{(3)}^*(y)\,dy
\stackrel{\rm def}{=}g(y;t,x,a)\times h(a-y;T-t,a)\ dy,\quad y\in (0,a)\ .
$$
As such, a very slight modification of Algorithm \ref{alg:interp:II} is sufficient and the full details are given in Algorithm \ref{alg:interp:III}. The results of our simulations for this case are displayed in Figure \ref{fig:interp:III}, where the information displayed is consistent with our previous simulations.

\begin{algorithm}[tbp]
  \caption{Exact Simulation of $Y_t^{x\to a}$.}
  \label{alg:interp:III}
  \begin{algorithmic}[1]
    \STATE Input $t,T,x,a$; 
    \STATE If $a/\sqrt{t}< 2$ set $C = C^{(1)}_g(t,x,a)$. Else, set $C = C^{(2)}_g(t,x,a)$;
    \STATE Simulate $Y\sim h(y;T-t,a)$; Set $Y:=a- Y$. Simulate $\sU$;
    \STATE If $\sU\times C\lerh g(Y;t,x,a)$, accept. Else, reject and return to STEP 3. 
    \STATE Output $Y$.
  \end{algorithmic}
\end{algorithm}

\begin{figure}[tbp]
    \centering
    \includegraphics[width=0.95\linewidth]{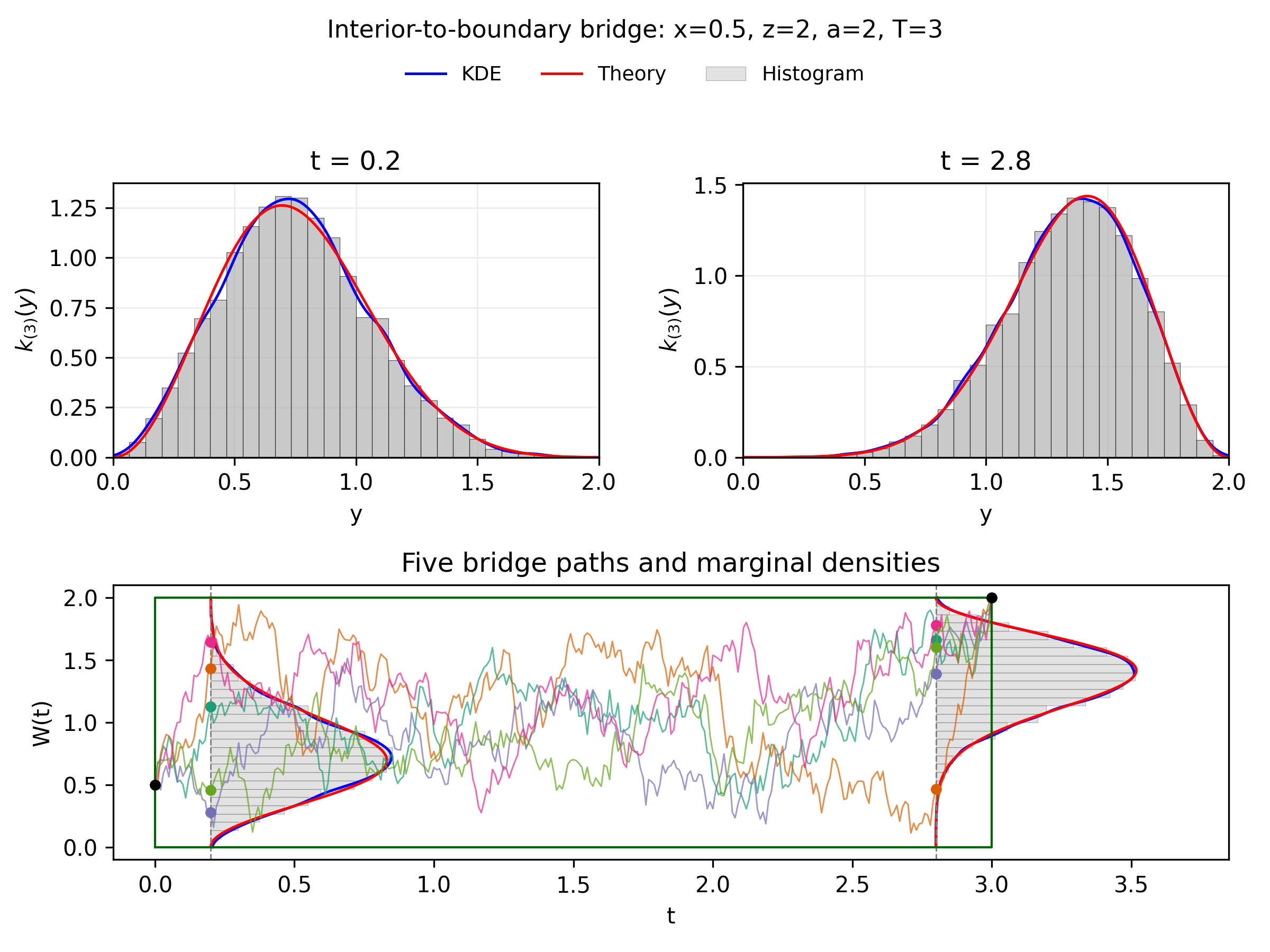}
    \caption{An illustration of Algorithm \ref{alg:interp:III}. The global parameters are set to $a=2, T=3, x=0.5, z=a$. Upper panels: marginal histograms (grey) based on $10,000$ joint draws of $(Y_{0.2}^{x\to z}, Y_{2.8}^{x\to z})$, kernel density estimates (blue) and theoretical densities (red), with $t=0.2$ on the left and $t=2.8$ on the right. The lower panel shows five separately simulated constrained Brownian bridge paths (various colors), with black dots marking the common start and end points, $(0,x)$ and $(T,z)$. The sideways marginal histograms (grey) use the same $10,000$ joint draws as the upper panels, with kernel density estimates (blue) and theoretical densities (red) superimposed.}
    \label{fig:interp:III}
\end{figure}

\noindent\textbf{Case III:} $0=x < z=a$. The argument used in Section \ref{sec:extrap} (Case II, $x=0$) can now be extended by letting $z\to a$. The density of $Y^{0\to a}_t$ is then proportional to 
$$
\PP(Y_t^{0\to a}\in dy) \propto k^*_{(4)}(y)\,dy = h(y;t,a)\times h(a-y;T-t,a)\,dy,\quad y\in (0,a)\ .
$$
Next, we develop an upper bound for $h$, to be used in our rejection sampling strategy. Combining the series bounds established in \eqref{eq:bound S1} and \eqref{eq:bound:S2} with the bounds on $h$ from \eqref{eq:bound:h1} and \eqref{eq:bound:h2}, we find that
\begin{enumerate}
    \item If $\alpha_t <\alpha_0$, we will use
    \begin{align*}
    h(y;t,a)&\le \sup_{y\in [0,a]}y\exp\left\{-\frac{y^2}{2t}\right\}  + 
    a \left[ \sum_{n=1}^{N_0+1} \exp\left\{-\beta_t n^2\right\}  + \exp\left\{-\beta_t (N_0+1)^2\right\}\right] \\
    & \stackrel{\rm def}{=}D^{(1)}_h(t,a)
    \end{align*}
    \item If $\alpha_t \ge \alpha_0$, then we will use
    \begin{align*}
    h(y;t,a)&\le 4\pi^{-1/2}\alpha_t^{3/2} \left[ \sum_{n=1}^{N_0+1} n^2\exp\left\{-\alpha_t n^2\right\} 
    +  (N_0+1)^2\exp\left\{-\alpha_t (N_0+1)^2\right\} \right]\frac{a}{2} \\
    &\stackrel{\rm def}{=}D^{(2)}_h(t, a)
    \end{align*}
\end{enumerate}
where we explicitly indicate that the bounds depend on $t$ (and $a$), which in turn will determine $\alpha_t= \pi^2 t/(2a^2)$ and $\beta_t=a^2/(2t)$. The value $N_0 = N_0(\alpha_t, \beta_t)$ is given in \eqref{eq:N0}. The supremum above can be found via standard calculus:
$$
\sup_{y\in [0,a]}y\exp\left\{-\frac{y^2}{2t}\right\} = \begin{cases}
    a\exp\{-a^2/(2t)\} & \mbox{ if } a < \sqrt{t} \\
    \sqrt{t}\exp\{-1/2\} & \mbox{ if } a \ge \sqrt{t}
\end{cases}
\ .
$$
We now have all the ingredients to describe the exact sampling procedure for $k^*_{(4)}$, which is given in Algorithm \ref{alg:interp:IV}. For any set of time points $0=t_0 < t_1 < \ldots <t_n <t_{n+1} =T$, an exact discretized path $\big(Y_{t_{1}}^{0\to a}, Y_{t_2}^{0\to a}, \ldots, Y_{t_n}^{0\to a}\big)$ can be simulated jointly in the following sequential way. Set $Y_{t_0}^{0\to a} = 0$, $Y_{t_{n+1}}^{0\to a} = a$. The variate $Y_{t_1}^{0\to a}$ is simulated using Algorithm \ref{alg:interp:IV} with inputs $t_1, T, a$. For $i=2,\ldots, n$, simulate $Y_{t_i}^{0\to a}$ using Algorithm \ref{alg:interp:III} with inputs
$$
\left( t_i - t_{i-1},\ T - t_{i-1},\ Y_{t_{i-1}}^{0\to a}, \,a \right)\ .
$$
The results of our simulations are displayed in Figure \ref{fig:interp:IV}, consistent with our previous visualizations. The upper panels show marginal histograms (grey) based on $10,000$ joint draws of $(Y_{0.2}^{0\to a}, Y_{2.8}^{0\to a})$, kernel density estimates (blue) and theoretical densities (red), with $t=0.2$ on the left and $t=2.8$ on the right. The lower panel shows five separately simulated constrained Brownian bridge paths from $x=0$ to $z=a$, with colored dots at $t=0.2$ and $t=2.8$ and black dots at the fixed endpoints. The sideways histograms and kernel density estimates use the same joint draws as the upper panels, with theoretical densities (red) superimposed.

\begin{algorithm}[tbp]
  \caption{Exact Simulation of $Y_t^{0\to a}$.}
  \label{alg:interp:IV}
  \begin{algorithmic}[1]
    \STATE Input $t,T,a$;  
    \STATE Set $\alpha_0 = 1.086165$; Evaluate $\alpha_{T-t}$, $\beta_{T-t}$, $N_0(\alpha_{T-t}, \beta_{T-t})$;
    \STATE If $\alpha_{T-t} < \alpha_0$ set $D = D^{(1)}_h(T-t,a)$. Else, set $D = D^{(2)}_h(T-t,a)$;
    \STATE Simulate $Y\sim h(y;t,a)$; Simulate $\sU$;
    \STATE If $\sU\times D\lerh h(a-Y;T-t,a)$, accept. Else, reject and return to STEP 3. 
    \STATE Output $Y$.
  \end{algorithmic}
\end{algorithm}

\begin{figure}[tbp]
    \centering
    \includegraphics[width=0.95\linewidth]{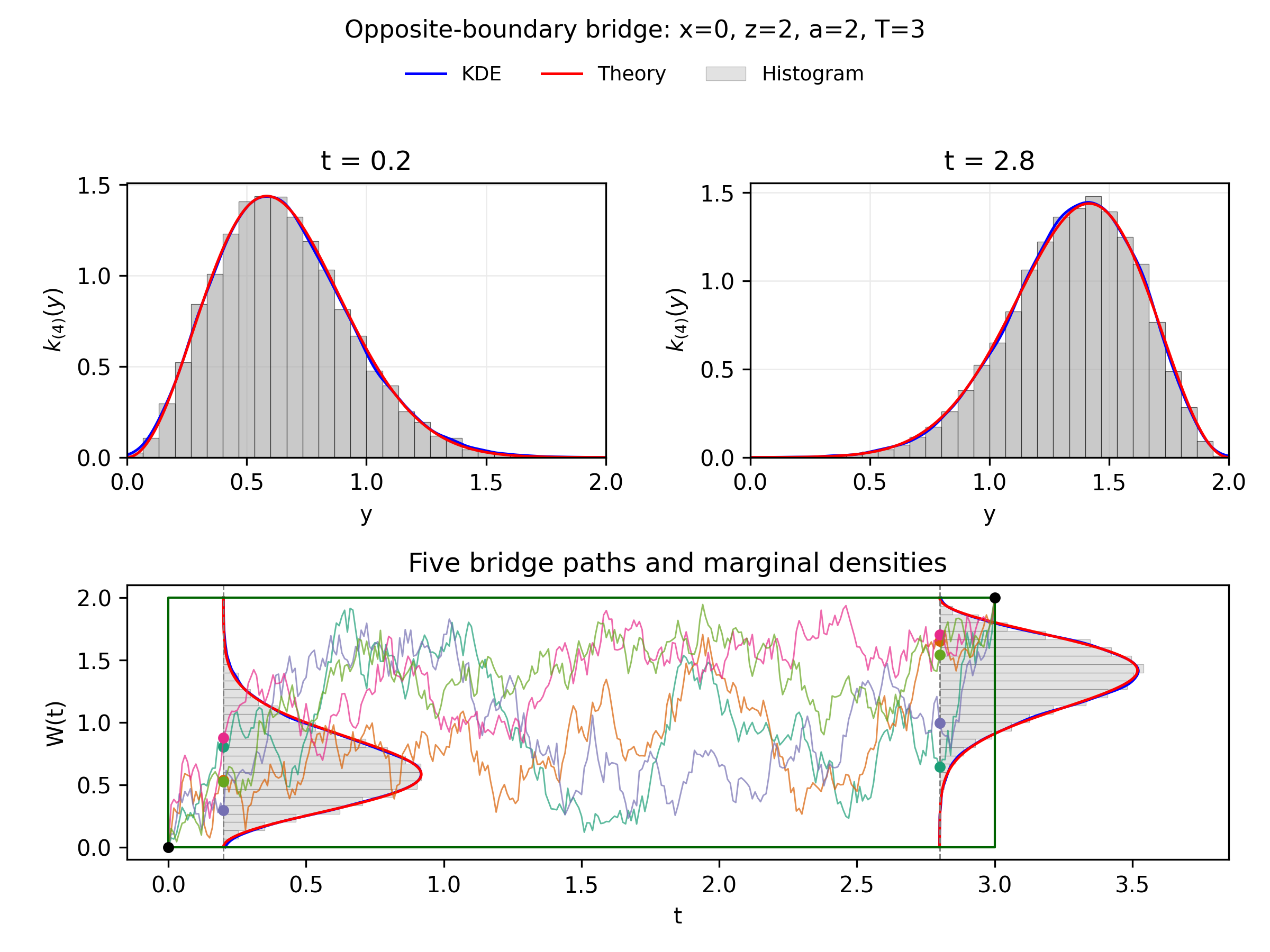}
    \caption{An illustration of Algorithm \ref{alg:interp:IV}. The global parameters are set to $a=2, T=3, x=0, z=a$. Upper panels: marginal histograms (grey) based on $10,000$ joint draws of $(Y_{0.2}^{x\to z}, Y_{2.8}^{x\to z})$, kernel density estimates (blue) and theoretical densities (red), with $t=0.2$ on the left and $t=2.8$ on the right. The lower panel shows five separately simulated constrained Brownian bridge paths (various colors), with black dots marking the common start and end points, $(0,x)$ and $(T,z)$. The sideways marginal histograms (grey) use the same $10,000$ joint draws as the upper panels, with kernel density estimates (blue) and theoretical densities (red) superimposed.}
    \label{fig:interp:IV}
\end{figure}

We end this section with the note that the two ``same-boundary'' cases, namely $0\to 0$ and $a\to a$ are straightforward varaints of Algorithm \ref{alg:interp:IV}, and we opt to omit the details.

\subsection{Some computational remarks}

From a computing perspective, at the core of our algorithms are the two series discussed in Proposition \ref{prop:h1}. Both series are absolutely convergent and tight bounds on the remainder terms are established. In particular, these bounds are what make the series comparisons in the accept/reject steps implementable
via monotone upper and lower partial sums, and they also provide a simple way to choose a truncation level when a numerical evaluation is required to a prescribed tolerance.

Specifically, the remainder terms are upper bounded by
$$
x^{(1)}_N(\alpha)=2(N+1)^2 \exp\{-\alpha(N+1)^2\}
\qquad \mbox{ and }\qquad
x^{(2)}_N(\beta)=2 \exp\{-\beta(N+1)^2\}\ ,
$$
where we used $\rho=1/2$. A simple analysis yields the following. Using $e^x \ge x$ for $x\ge 0$, with $x=\alpha(N+1)^2/2$ we find
\[
2(N+1)^2 \exp\{-\alpha(N+1)^2\} \le \frac{4}{\alpha}\exp\{-\alpha(N+1)^2/2\}.
\]
Thus, for a given tolerance $\texttt{tol}$, we find that
\[
N > \left\lceil \sqrt{\frac{2}{\alpha}\log\!\left(\frac{4}{\alpha\times \texttt{tol}}\right)} \right\rceil -1
\ \Rightarrow\ x^{(1)}_N(\alpha) < \texttt{tol}.
\]
Similarly,
\[
N > \left\lceil \sqrt{\frac{1}{\beta}\log\!\left(\frac{2}{\mathrm{tol}}\right)} \right\rceil -1
\ \Rightarrow\ x^{(2)}_N(\beta) < \texttt{tol}.
\]

Figure \ref{fig:tol} depicts the continuous square-root expressions in these conditions (on a $\log_{10}$ scale), before integer rounding, for $\texttt{tol}=10^{-16}$. Each parameter ranges independently from $10^{-3}$ to $10$, with the $\beta$ axis reversed for display. In our application, the parameters $(\alpha_t,\beta_t)$ satisfy $\beta_t=\pi^2/(4\alpha_t)$; the two axes in the figure do not represent this reciprocal pairing.
At least one of these quantities is large, and in combination with the regime-switching used throughout Section \ref{sec:sampling} this leads to rapid
convergence in practice. The conditions above should be interpreted as sufficient (and typically
conservative) guarantees that the tail is below $\texttt{tol}$; in the series method, the accept/reject decision is often resolved after fewer terms. The required $N$ grows only in the small-$\alpha$ (resp.\ small-$\beta$) regime, roughly like
$\sqrt{\alpha^{-1}\log(1/\texttt{tol})}$ (resp.\ $\sqrt{\beta^{-1}\log(1/\texttt{tol})}$). 
\begin{figure}[tbp]
    \centering
    \includegraphics[width=\linewidth]{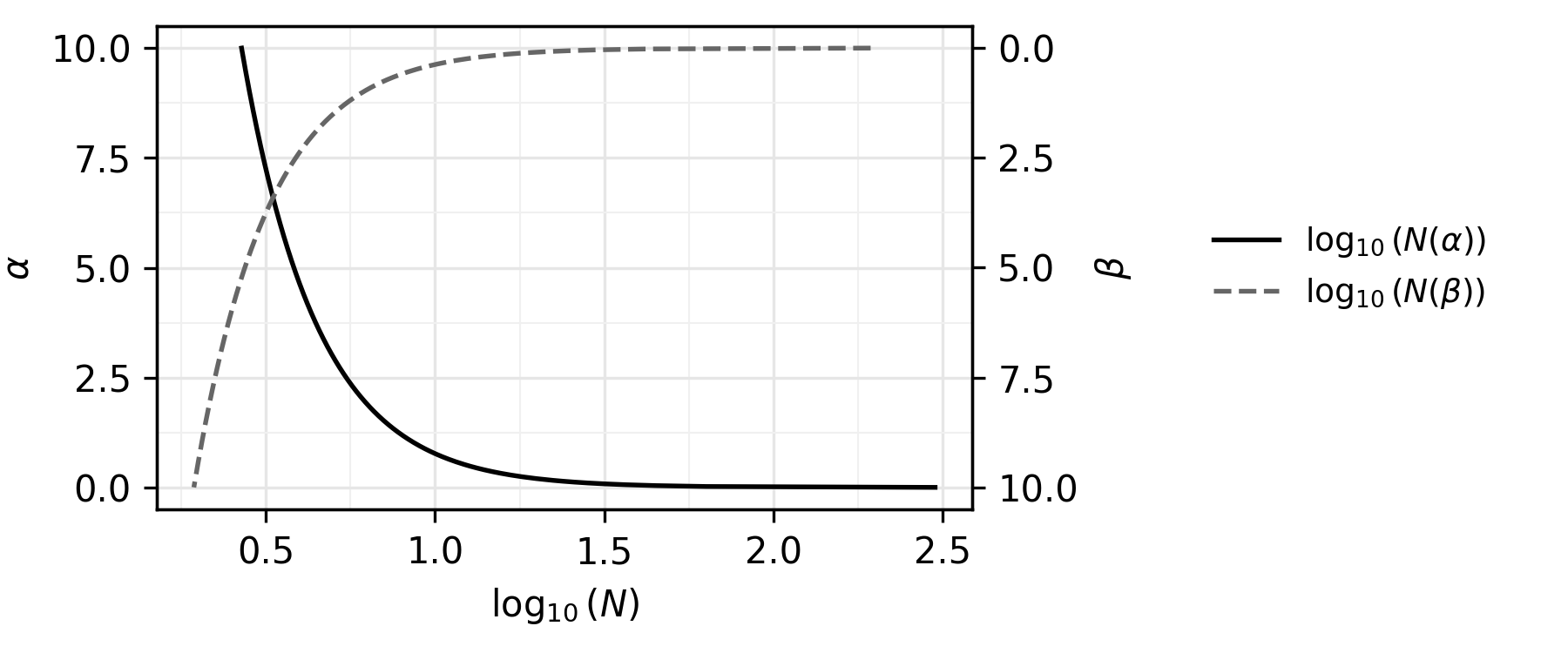}
    \caption{Continuous expressions underlying the displayed truncation conditions, shown on a $\log_{10}$ scale before integer rounding, for $\mathtt{tol}=10^{-16}$. The solid black curve corresponds to the $\alpha$ expression and the dashed grey curve to the $\beta$ expression. The $\beta$ axis is reversed for display.}
    \label{fig:tol}
\end{figure}
We note that a very modest $N$-value is sufficient to achieve machine precision, which we mean to be interpreted
as the IEEE 754 double precision: the unit round-off is $2^{-53}\approx 1.11\times 10^{-16}$. Smaller tolerances
cannot be meaningfully resolved in double precision arithmetic.

The acceptance probabilities of Algorithms~\ref{alg:interp:I}--\ref{alg:interp:IV} admit explicit
series representations. For fixed $a>0$, define the normalizing
constants
\begin{align*}
G_t(x)
&:= \int_0^a g(y;t,x,a)\,dy
 = \frac{4}{\pi}
   \sum_{\substack{n\ge1\\ n\text{ odd}}}
   \frac{\sin(n\pi x/a)}{n}e^{-\alpha_t n^2},\\
H_t
&:= \int_0^a h(y;t,a)\,dy
 = \frac{2\sqrt{2\pi}\,t^{3/2}}{a}
   \sum_{\substack{n\ge1\\ n\text{ odd}}}
   e^{-\alpha_t n^2}.
\end{align*}
Let $A_j$ denote the probability of accepting a proposal at the
outer rejection step of Algorithm $j$. Integrating each acceptance
test against its normalized proposal density gives
\begin{align*}
A_3
&= \frac{g(z;T,x,a)}{C\,G_t(x)},\\
A_4
&= \frac{(t/T)^{3/2}h(z;T,a)}{C\,H_t},\\
A_5
&= \frac{((T-t)/T)^{3/2}h(a-x;T,a)}{C\,H_{T-t}},\\
A_6
&= \frac{\pi^3[t(T-t)]^{3/2}}{a^3D\,H_t}
   \sum_{n=1}^{\infty}
   (-1)^{n+1}n^2e^{-\alpha_T n^2}.
\end{align*}
Here $C$ is the constant selected in the respective algorithm,
$D$ is the constant selected in Algorithm \ref{alg:interp:IV}, and
$\alpha_T=\pi^2T/(2a^2)$. These identities follow from the
series representations \eqref{eq:g2} and \eqref{eq:h2} by termwise integration
and orthogonality of the sine functions. They exclude the
internal rejections required to generate each proposal and
do not measure the cost of the series comparisons.

We end this section with the following observation and experiment. For clarity of exposition, Algorithms~\ref{alg:interp:I}--\ref{alg:interp:IV}
are presented in their basic forms, with a fixed choice of proposal
orientation. In our Python implementation, we additionally exploit the
time-reversal symmetry of the survival-conditioned Brownian bridge to
select the proposal associated with the shorter of the two time intervals,
$t$ and $T-t$. This entails interchanging the roles of the endpoints and,
where required by the boundary configuration, applying the spatial
reflection $y\mapsto a-y$, with accepted draws subsequently expressed in
the original coordinates. The remaining factor is evaluated in the outer
rejection step over the longer time interval, using the corresponding
envelope and series bounds. This deterministic choice is made before
generating proposals and preserves the target distribution. For a density
factor associated with an interval of length $s$, the implementation uses
the image bounds when $\alpha_s<\pi$ and the spectral bounds when
$\alpha_s\geq\pi$, where $\alpha_s=\pi^2s/(2a^2)$. In the spectral
regime, assigning the longer interval to the outer comparison promotes
faster decay of the higher modes after removing the common factor
$\exp\{-\alpha_s\}$. The proposal-orientation rule is distinct from this
choice of representation and is not intended to maximize the outer
acceptance probability. The displayed pseudocode and
associated acceptance-probability expressions refer to the basic proposal
orientations, whereas all bridge simulations and timing experiments reported
in this paper were conducted using the implementation with automatic
time reversal.

We complement these expressions with a Monte Carlo experiment and report
the empirical results in Table~\ref{tab:mc}. We fix $a=2$, $T=3$, and
$t\in\{0.1,1,1.5,2,2.5,2.9\}$. Each interior endpoint ranges over
$\{0.05,0.525,1,1.475,1.95\}$, giving $25$ endpoint settings for
Algorithm~\ref{alg:interp:I}, five each for
Algorithms~\ref{alg:interp:II} and~\ref{alg:interp:III}, and one for
Algorithm~\ref{alg:interp:IV}, at each value of $t$. For each setting,
we estimate the outer acceptance probability as $2,000$ divided by the
number of proposals required to obtain $2,000$ accepted draws. We then
time a separate batch of $10,000$ draws with proposal counting disabled.
The table reports the arithmetic mean of the acceptance estimates and
the median of the batch times across endpoint settings. For
Algorithm~\ref{alg:interp:IV}, each entry corresponds to its single
endpoint setting. All bridge simulations use automatic time reversal
and the image/spectral comparisons described above, and were run serially
in Python~3.9.6 on a laptop running macOS.

\begin{table}[!ht]
\centering
\caption{Mean outer acceptance rates and median elapsed times (seconds per $10,000$ draws) across endpoint settings for Algorithms \ref{alg:interp:I}--\ref{alg:interp:IV}, with automatic time reversal and image/spectral comparisons.}
\label{tab:mc}
\begin{tabular}{c|cc|cc|cc|cc}
\hline
& \multicolumn{2}{c|}{Algorithm \ref{alg:interp:I}} & \multicolumn{2}{c|}{Algorithm \ref{alg:interp:II}} & \multicolumn{2}{c|}{Algorithm \ref{alg:interp:III}} & \multicolumn{2}{c}{Algorithm \ref{alg:interp:IV}} \\
$t$ & Acc. & Med. time & Acc. & Med. time & Acc. & Med. time & Acc. & Med. time \\
\hline
0.1 & 0.113 & 1.064 & 0.089 & 1.596 & 0.437 & 0.350 & 0.363 & 0.489 \\
1.0 & 0.130 & 1.264 & 0.130 & 1.763 & 0.506 & 0.375 & 0.493 & 0.485 \\
1.5 & 0.129 & 2.483 & 0.128 & 1.591 & 0.128 & 1.599 & 0.495 & 0.477 \\
2.0 & 0.130 & 1.256 & 0.503 & 0.376 & 0.130 & 1.686 & 0.512 & 0.486 \\
2.5 & 0.128 & 1.201 & 0.492 & 0.315 & 0.129 & 1.265 & 0.488 & 0.322 \\
2.9 & 0.112 & 1.064 & 0.438 & 0.345 & 0.090 & 1.639 & 0.343 & 0.491 \\
\hline
\end{tabular}
\end{table}

The median batch times in Table~\ref{tab:mc} correspond to approximately
$0.031$--$0.248$ milliseconds per draw. The contrasting acceptance patterns
of Algorithms~\ref{alg:interp:II} and~\ref{alg:interp:III} primarily
reflect their changes of proposal orientation at $t=T/2$. For
Algorithm~\ref{alg:interp:II}, the implementation proposes from $h$
and tests against $g$ when $t\le T/2$, whereas it proposes from $g$
and tests against $h$ when $t>T/2$. Algorithm~\ref{alg:interp:III}
uses the $g$ proposal and a reflected $h$ factor when $t<T/2$, and
a reflected $h$ proposal with a $g$ test when $t\ge T/2$.
For the parameters considered here, the relatively conservative
constant envelope for $g$ contributes to the lower endpoint-averaged
acceptance rates in the $h$-proposal/$g$-test branches.
Thus, Algorithm~\ref{alg:interp:II} has
higher acceptance after the midpoint, while
Algorithm~\ref{alg:interp:III} has higher acceptance before it.
At $t=T/2$, both implementations retain the $h$ proposal and have
similar acceptance rates. Moreover, near a boundary endpoint, the
short-time $h$ proposal concentrates near that boundary, where the
remaining $g$ factor vanishes; this contributes to the lower acceptance
for Algorithm~\ref{alg:interp:II} at small $t$ and for
Algorithm~\ref{alg:interp:III} at $t$ close to $T$.
These observations describe the present parameter configuration and
do not assert monotonicity in $t$ or a uniformly preferable orientation
over other values of $a$, $T$, and the endpoints. Acceptance probabilities
also exclude the cost of generating proposals and evaluating the series
bounds, so they do not alone determine elapsed time.

Algorithms~\ref{alg:interp:I} and~\ref{alg:interp:IV} retain the same
proposal family on both sides of the midpoint: respectively, a $g$
proposal with a $g$ test, and an $h$ proposal with a reflected $h$ test.
The approximate symmetry of their reported acceptance rates about
$T/2$ reflects time reversal, together with averaging over the full
endpoint grid for Algorithm~\ref{alg:interp:I}. The acceptance rate
of Algorithm~\ref{alg:interp:I} remains comparatively stable over the
reported times. Algorithm~\ref{alg:interp:IV} has lower acceptance
near either endpoint, where the short-time $h$ proposal concentrates
near a boundary at which the remaining, reflected $h$ factor vanishes.

\section{Conclusion}
\label{sec:concl}

This work provides an exact way to extrapolate and interpolate Brownian paths restricted to a compact interval $[0,a]$. Standard techniques can be used to extend our algorithms to any other compact interval. At the core of our approach are the two complementary pictures of the probability density of interest. This leads us to two families of simple proposal distributions and one accept-reject engine built on monotone series bounds. We develop a set of algorithms which work across different regimes, cover interior and boundary endpoints in one sweep, and can be easily extended to construct skeleton paths via the Markov property. The dual representation allows us to develop algorithms with high acceptance rates across regimes.

The exact samplers derived here are variants of Devroye's series method. We work with absolutely convergent series, for which we develop tight remainder bounds, allowing us to construct monotone upper/lower partial sums. A straightforward computational analysis demonstrates the practicality of our approach. Machine-precision accuracy is attainable with modest truncation. Our algorithms are almost universally fast, and are built on principled automation. Rather than relying on heuristics, a scale-free summary of time versus interval width selects between the reflection and eigenfunction views, providing the tighter envelope.

From an applied perspective, the procedures are meant to be drop-in. Single time draws or skeleton paths are obtained by calling the same kernel, with boundary configurations handled automatically. The constants and arguments are selected to keep the implementation straightforward and portable, facilitating integration into larger Monte Carlo schemes.

In sum, our work shows that exactness need not come at the expense of practicality. By combining classical representations with carefully controlled series bounds the paper provides exact algorithms for extrapolation and interpolation under survival that are fast, robust and straightforward to implement. Our routines are intended to serve as reliable building blocks for simulation and inference tasks involving Brownian paths.

\begin{appendices}

\section{Jacobi Theta function identity}
\label{seec:jacobi}
For $y\in \RR$, $z>0$, the Jacobi theta identity \cite{NISTHandbook2010} states that 
    \begin{equation}
    \label{eq:jacobi1}
    \frac{1}{\sqrt{\pi z}}\sum_{k=-\infty}^{\infty}\exp\left\{-\frac{(k+y)^2}{z} \right\} \hspace{2mm}= \sum_{k=-\infty}^{\infty}\cos(2\pi ky)\exp\left\{-k^2\pi^2z\right\}.
    \end{equation}

\section{Proof of Proposition \ref{prop:markov}}
\label{proof:markov}

Let $D = (0,a)$ and $\tau = \inf\{s\ge 0\ :\ W_s \not\in D\}$ be the exit time from $D$. We use $\PP_x, \EE_x$ to denote the probability measure and expected value respectively as induced by $W^x$. Let $p^D_t(u,v)$ be the transition density of the Brownian motion killed on exiting $D$ \cite{borodin2012handbook,ito2012diffusion,revuz2013continuous}, that is
$$
p_t^D(u,v)dv = \PP_u(W_t\in dv,\ \tau >t)\ .
$$
We also define the post-$t$ exit time $\tau^{(t)} = \inf\{s\ge 0\ :\ W_{t+s}\not\in D\}$, and we note that for $0<t<T$, we have
$$
\bfone\{\tau > T\} = \bfone\{\tau>t\}\bfone\{\tau^{(t)} > T-t\}\ .
$$ 
For two measurable functions $f,g$ we have
\begin{align*}
\EE_x\left[f(W_t)g(W_T)\bfone\{\tau > T\}\right] &= 
\EE_x\left[ f(W_t) g(W_T)\bfone\{\tau>t\}\bfone\{\tau^{(t)} > T-t\}\right] \\
&= \EE_x\left[f(W_t)\bfone\{\tau>t\}\EE_x\left[g(W_T)\bfone\{\tau^{(t)}>T-t\}  \mid W_t\right]\right]
\end{align*}
Markov property gives
$$
\EE_x\left[g(W_T)\bfone\{\tau^{(t)}>T-t\}  \mid W_t=y\right] = \EE_y\left[g(W_{T-t})\bfone\{\tau > T-t\}\right]
$$
For two Borel sets $B_1, B_2\subset (0,a)$, take $f = \bfone\{y\in B_1\}$, $g = \bfone\{z\in B_2\}$. It follows that
$$
\PP_x(W_t\in B_1, W_T\in B_2, \tau > T) = \int_{B_1}\int_{B_2} p_t^D(x,y)p_{T-t}^D(y,z)\,dydz
$$
or, equivalently
$$
\PP_x(W_t\in dy, W_T\in dz, \tau > T) =  p_t^D(x,y)p_{T-t}^D(y,z)\,dydz
$$
Divide by $\PP_x(W_T\in dz, \tau > T) = p_T^D(x,z)dz$  and use the fact that $p_t^D(u,v) = p_t^D(v,u)$ to obtain
$$
\PP(W^{x\to z}_t\in dy \mid B) = \frac{p_t^D(x,y)p^D_{T-t}(z,y)}{p_T^D(x,z)}dy
$$
Finally, note that
$$
\PP(W_t^x \in dy\mid A_x(t)) = \frac{p_t^D(x,y)}{\PP_x(\tau >t)}dy
\quad \mbox{and} \quad
\PP(W_{T-t}^z \in dy\mid A(T-t)) = \frac{p_{T-t}^D(z,y)}{\PP_z(\tau >T-t)}dy
$$
which gives the desired relation.

\section{Proof of Theorem \ref{thm:bound:h}}
\label{proof:bound:h}

We prove the first inequality using the expansion given in \eqref{eq:h1}. Note that, since the series defining $h$ converge absolutely due to the exponential damping term, we can write
\begin{align*}
h(y;t,a) =\sum_{n=-\infty}^{\infty}y\exp\left\{\frac{-(y+2na)^2}{2t}\right\}
+
2a\sum_{n=-\infty}^{\infty}n\exp\left\{\frac{-(y+2na)^2}{2t}\right\}
\end{align*}
Since $y>0$, the first term above is positive. Write the second term as
$$
\sum_{n=-\infty}^{\infty}
n\exp\left\{\frac{-(y+2na)^2}{2t}\right\}
=
\sum_{n=1}^{\infty}n\left[\exp\left\{\frac{-(2na+y)^2}{2t}\right\}-\exp\left\{\frac{-(2na-y)^2}{2t}\right\}\right]
$$
Since $y\in (0,a)$, for $n\ge 1$ we have
$$
(2n+1)a >2na+y > 2na > 2na-y > (2n-1)a \ge  a>0,
$$
Thus
$$
\exp\left\{\frac{-(2na+y)^2}{2t}\right\}-\exp\left\{\frac{-(2na-y)^2}{2t}\right\} < 0
$$
We conclude that
\begin{align*}
h(y;t,a) & \le \sum_{n=-\infty}^{\infty}y\exp\left\{\frac{-(y+2na)^2}{2t}\right\} \\
&= y\exp\left\{\frac{-y^2}{2t}\right\} + y \sum_{n=1}^{\infty}\left[\exp\left\{\frac{-(2na+y)^2}{2t}\right\}+\exp\left\{\frac{-(2na-y)^2}{2t}\right\}\right]
\end{align*}
Now
$$
\begin{aligned}
\exp\left\{\frac{-(2na+y)^2}{2t}\right\}
&< \exp\left\{\frac{-(2na)^2}{2t}\right\},\\
\exp\left\{\frac{-(2na-y)^2}{2t}\right\}
&< \exp\left\{\frac{-((2n-1)a)^2}{2t}\right\}.
\end{aligned}
$$
Thus,
\begin{align*}
    h(y;t,a) &\le y\exp\left\{\frac{-y^2}{2t}\right\} +y
    \sum_{n=1}^{\infty} \left[\exp\left\{\frac{-(2na)^2}{2t}\right\}+\exp\left\{\frac{-((2n-1)a)^2}{2t}\right\}\right] \\
    &=  y\exp\left\{\frac{-y^2}{2t}\right\} +y \sum_{n=1}^{\infty}\exp\left\{\frac{-n^2a^2}{2t}\right\}.
\end{align*}
For the second inequality, we use the sin expansion \eqref{eq:h2}. Using the fact that:
$$
\left|\sin\left(\frac{n\pi y}{a}\right)\right|  = \left|\sin\left(\frac{n\pi (a-y)}{a}\right)\right| \le \frac{n\pi}{a}\min\{y, a-y\}
$$
we obtain
\begin{align*}
h(y;t,a)&\le \frac{\sqrt{2}\pi^{5/2} t^{3/2}}{a^3}\sum_{n=1}^\infty n^2\exp\{-\alpha_t n^2\} \times \min\{ y,a-y\} \\
& = 4\pi^{-1/2}\alpha_t^{3/2} \sum_{n=1}^\infty n^2\exp\{-\alpha_t n^2\} \times \min\{ y,a-y\}\ .
\end{align*}

\section{Proof of Proposition \ref{prop:h1}}
\label{proof:h1}

We give a complete proof for the bounds and the properties related to $S^{(1)}$. This proof can be reproduced line-by-line to establish the bounds and approximating sequences for $S^{(2)}$. We note the following equivalences
\begin{align*}
&\frac{(n+1)^2\exp\{-\alpha(n+1)^2\}}{n^2\exp\{-\alpha n^2\}} < \rho \qquad \Leftrightarrow \\
&\frac{1}{\rho}\frac{(n+1)^2}{n^2} < \frac{\exp\{\alpha(n+1)^2\} }{\exp\{\alpha n^2\}} = \exp\{\alpha(2n+1)\} \qquad \Leftrightarrow\\
&\log(1/\rho) + 2\log(1+1/n) < \alpha(2n+1)
\end{align*}
The left-hand side term is upper-bounded by
$$
\log(1/\rho) + 2\log(1+1/n) \le \log(4/\rho)
$$
and finally
$$
\log(4/\rho) < \alpha(2n+1)
\quad
\Leftrightarrow\quad
n> \frac{1}{2}\left[\frac{1}{\alpha}\log\left(\frac{4}{\rho}\right)-1\right]\stackrel{\rm def}{=} N_0(\alpha, \rho)
$$
Thus, if $N > N_0(\alpha,\rho)$ we have
$$
R^{(1)}_N < (N+1)^2\exp\{-\alpha(N+1)^2\}\sum_{n\ge 0}\rho^n = (N+1)^2\exp\{-\alpha(N+1)^2\}/(1-\rho)\ .
$$
This proves the first part of the proposition. For the second part, note that the terms in the series defining $S^{(1)}(\alpha)$ are positive, thus $S^{(1)}_N(\alpha)\nearrow S^{(1)}(\alpha)$ as $N\to \infty$. Since $x^{(1)}_N(\alpha)\to 0$, it also follows that $S^{(1)}_N(\alpha) + x^{(1)}_N(\alpha)\to S$ as $N\to \infty$. It remains to show that $S^{(1)}_N(\alpha) + x^{(1)}_N(\alpha)$ is a decreasing sequence. For this, note that
\begin{align*}
   S^{(1)}_{N+1}(\alpha) + x^{(1)}_{N+1}(\alpha) &\le S^{(1)}_N(\alpha) + x^{(1)}_N(\alpha) \qquad \Leftrightarrow \\
 S^{(1)}_{N+1}(\alpha) - S^{(1)}_N(\alpha) &\le x^{(1)}_N(\alpha) - x^{(1)}_{N+1}(\alpha)\qquad \Leftrightarrow \\
 (1-\rho)x^{(1)}_N(\alpha) &\le x^{(1)}_N - x^{(1)}_{N+1}(\alpha) \qquad \Leftrightarrow \\
x^{(1)}_{N+1}(\alpha) &\le \rho x^{(1)}_N(\alpha)\ ,
\end{align*}
which holds for $N\ge N_0(\alpha, \rho)$ as demonstrated above. 

Analyzing the second series in a similar way yields that the desired inequalities hold for 
$$
N\ge N_0(\beta,\rho) = \frac{1}{2}\left[\frac{1}{\beta}\log\left(\frac{1}{\rho}\right)-1\right]\ .
$$
Combining these two cases into one, gives the desired result.

\section{Proof of Proposition \ref{prop:g2} and Proposition \ref{prop:h2}}
\label{proof:h2}

We begin by providing a complete proof of Proposition \ref{prop:h2}. At the end of this section we discuss a few minor modifications needed for Proposition \ref{prop:g2}. As established in the main text, $|R^h_N(y)|\le x^h_N(y)$, and thus
\begin{align*}
S^h_N(y) + x^h_N(y) &\ge S_N^h(y) + R^h_N(y) = h(y;t,a)
\qquad\mbox{ and }\\
S^h_N(y) -x^h_N(y) &\le S^h_N(y) + R^h_N(y) = h(y;t,a)
\end{align*}
Next we establish the monotonicity of the two upper and lower bounding sequences. We drop the dependence on $y$, to simplify the notation. Note the following equivalence: 
\begin{align*}
    S_N^h - x_N^h &\le S_{N+1}^h - x^h_{N+1}\qquad \Leftrightarrow\\
    x_{N+1}^h - x_N^{h} &\le K^h (N+1)\sin\left(\frac{\pi(N+1)y}{a}\right)
    \exp\left\{-\frac{\pi^2(N+1)^2t}{2a^2}\right\} \ .
\end{align*}
Now, for $N > N_0(\alpha_t, \beta_t, \rho)$ as in Proposition \ref{prop:h1}, we have 
$$
 x^h_{N+1} \le \rho x^h_{N}\ .
$$
Thus, using $-x\le \sin(x)$ for $x\ge 0$, we have
\begin{align*}
x_{N+1}^h - x_N^h &\le -(1-\rho) x_{N}^h = -K^h \frac{\pi y}{a} (N+1)^2 \exp\left\{-\frac{\pi^2(N+1)^2t}{2a^2} \right\} \\
&\le K^h (N+1) \sin\left(\frac{\pi(N+1)y}{a}\right)\exp\left\{-\frac{\pi^2(N+1)^2t}{2a^2}\right\}
\end{align*}
as required. Conclude that $S_N^h - x_N^h$ is eventually increasing. Similarly, we have the following equivalence: 
\begin{align*}
    S_N^h + x_N^h &\ge S_{N+1}^h + x^h_{N+1}\qquad \Leftrightarrow\\
    -x_{N+1}^h + x_N^{h} &\ge K^h (N+1)\sin\left(\frac{\pi(N+1)y}{a}\right)
    \exp\left\{-\frac{\pi^2(N+1)^2t}{2a^2}\right\} 
\end{align*}
Again, for large $N$ we have $x_{N+1}^h\le \rho x_N^h$ and thus, using $\sin(x) \le x$ for $x\ge 0$, we have
\begin{align*}
-x_{N+1}^h + x_N^h &\ge (1-\rho) x_{N}^h = K^h \frac{\pi y}{a} (N+1)^2 \exp\left\{-\frac{\pi^2(N+1)^2t}{2a^2} \right\} \\
&\ge K^h (N+1) \sin\left(\frac{\pi(N+1)y}{a}\right)\exp\left\{-\frac{\pi^2(N+1)^2t}{2a^2}\right\}\ ,
\end{align*}
as required. Finally, since $x_N^h(y)\to 0$ as $N\to \infty$, it follows that
$$
S_N^h(y) + x_N^h(y)\searrow h(y;t,a) 
\qquad \mbox{ and }\qquad 
S_N^h(y) - x_N^h(y)\nearrow h(y;t,a)\ .
$$
as $N\to \infty$.

The proof of Proposition \ref{prop:g2} can be reproduced from above, line by line. In the monotonicity argument, the following inequality is used instead:
$$
\forall\ x,y\ge 0 \qquad |\sin(x)\sin(y)| \le xy
\quad\Leftrightarrow\quad -xy \le \sin(x)\sin(y)\le xy\ .
$$
This inequality follows directly from 
$
|\sin(x)\sin(y)| = |\sin(x)|\ |\sin(y)| \le xy\ .
$

\section{Proofs of Propositions~\ref{prop:image-g}
and~\ref{prop:image-h}}
\label{app:image-bounds}

Both image series converge absolutely and uniformly in $y\in[0,a]$.
Indeed, for $|n|\geq2$,
$|y+2na\pm x|\geq a|n|$ and $|2na+y|\geq a|n|$,
whereas $|2na+y|\leq a(2|n|+1)$. Thus, apart from constant
factors, the absolute values of their summands are bounded by
$(1+|n|)\exp\{-\beta_tn^2\}$, a summable sequence.
Pairing the terms with indices $n$ and $-n$ is therefore legitimate,
including at $y=0$ and $y=a$.

\begin{proof}[Proof of Proposition~\ref{prop:image-g}]
Fix $a,t>0$, $0<x<a$, and $y\in[0,a]$. For $n\geq1$, put
\[
\Delta_n^g(y)=\widetilde S_n^g(y)-\widetilde S_{n-1}^g(y),
\qquad d_n=2na-x-y>0,
\qquad F(u)=\exp\{-u^2/(2t)\},\quad u\geq0.
\]
Pairing the two new image terms gives
\begin{align*}
\sqrt{2\pi t}\,\Delta_n^g(y)
&=F(d_n+2y)-F(d_n+2x+2y)
  -F(d_n)+F(d_n+2x)\\
&=\bigl[F(d_n+2y)-F(d_n+2x+2y)\bigr]
  -\bigl[F(d_n)-F(d_n+2x)\bigr].
\end{align*}
Since $F$ is nonnegative and decreasing on $[0,\infty)$,
both bracketed expressions belong to $[0,F(d_n)]$. Consequently,
\begin{equation}
|\Delta_n^g(y)|\leq b_n,
\qquad\mbox{ where we set }\quad
b_n=\frac{\exp\{-d_n^2/(2t)\}}{\sqrt{2\pi t}}.
\label{eq:image-g-increment}
\end{equation}
The consecutive ratios satisfy
\[
r_n:=\frac{b_{n+1}}{b_n}
=\exp\left\{-\frac{2a[(2n+1)a-x-y]}{t}\right\},
\qquad 0<r_{n+1}<r_n<1.
\]
In particular,
$\widetilde x_N^g(y)=b_{N+1}/(1-r_{N+1})$.
Absolute convergence, \eqref{eq:image-g-increment}, and the decrease
of $r_n$ imply
\begin{align*}
|\widetilde R_N^g(y)|
&\leq\sum_{j=0}^{\infty}b_{N+1+j}
\leq b_{N+1}\sum_{j=0}^{\infty}r_{N+1}^{\,j}
=\frac{b_{N+1}}{1-r_{N+1}}
=\widetilde x_N^g(y).
\end{align*}
This proves the two bounds in part~1. Next, using $r_{N+2}\leq r_{N+1}$, we obtain
\[
\widetilde x_{N+1}^g(y)
=\frac{r_{N+1}b_{N+1}}{1-r_{N+2}}
\leq r_{N+1}\widetilde x_N^g(y).
\]
Therefore,
\begin{equation}
\widetilde x_N^g(y)-\widetilde x_{N+1}^g(y)
\geq (1-r_{N+1})\widetilde x_N^g(y)
=b_{N+1}
\geq|\Delta_{N+1}^g(y)|.
\label{eq:image-g-decrement}
\end{equation}
It follows that
\begin{align*}
&[\widetilde S_{N+1}^g(y)-\widetilde x_{N+1}^g(y)]
-[\widetilde S_N^g(y)-\widetilde x_N^g(y)]
\geq\Delta_{N+1}^g(y)+|\Delta_{N+1}^g(y)|\geq0,\\
&[\widetilde S_{N+1}^g(y)+\widetilde x_{N+1}^g(y)]
-[\widetilde S_N^g(y)+\widetilde x_N^g(y)]
\leq\Delta_{N+1}^g(y)-|\Delta_{N+1}^g(y)|\leq0.
\end{align*}
Finally, $b_{N+1}\to0$ and $r_{N+1}\to0$, so that
$\widetilde x_N^g(y)\to0$ as $N\to \infty$. Together with part~1, this establishes
the two monotone limits.
\end{proof}

\begin{proof}[Proof of Proposition~\ref{prop:image-h}]
Fix $a,t>0$ and $y\in[0,a]$, and define
\[
q(u)=u\exp\{-u^2/(2t)\},\qquad u\geq0.
\]
and note that $q$ is nonnegative and decreasing on
$[\sqrt t,\infty)$. Pairing the terms with indices $n$ and $-n$
in the image series gives
\[
\widetilde S_N^h(y)
=q(y)+\sum_{n=1}^{N}\bigl[q(2na+y)-q(2na-y)\bigr].
\]
Let $N\geq\widetilde N_0(\beta_t)$. By
\eqref{eq:image-cutoff},
\[
(2N+1)a\geq\sqrt t.
\]
Hence, for every $n\geq N+1$, all the following arguments are
in $[\sqrt t,\infty)$, and $0\leq y\leq a$ gives
\[
2na-y\leq2na+y\leq2(n+1)a-y.
\]
The monotonicity of $q$ therefore yields
\begin{equation}
0\leq q(2na-y)-q(2na+y)
\leq q(2na-y)-q(2(n+1)a-y).
\label{eq:image-h-telescoping}
\end{equation}
For $M>N$, summing these inequalities gives
\begin{align*}
0\leq\widetilde S_N^h(y)-\widetilde S_M^h(y)
&=\sum_{n=N+1}^{M}
  \bigl[q(2na-y)-q(2na+y)\bigr]\\
&\leq q(2(N+1)a-y)-q(2(M+1)a-y).
\end{align*}
Letting $M\to\infty$, and using absolute convergence and
$q(u)\to0$ as $u\to\infty$, we obtain
\begin{equation}
0\leq\widetilde S_N^h(y)-h(y;t,a)
\leq q(2(N+1)a-y)
=\widetilde x_N^h(y).
\label{eq:image-h-remainder}
\end{equation}
This proves part~1 and the sharper upper bound. To establish monotonicity, write
$\Delta_{N+1}^h(y)=\widetilde S_{N+1}^h(y)-\widetilde S_N^h(y)$.
Taking $n=N+1$ in \eqref{eq:image-h-telescoping} shows that
$\Delta_{N+1}^h(y)\leq0$ and
\[
\widetilde x_N^h(y)-\widetilde x_{N+1}^h(y)
\geq|\Delta_{N+1}^h(y)|.
\]
Consequently,
\begin{align*}
&[\widetilde S_{N+1}^h(y)-\widetilde x_{N+1}^h(y)]
-[\widetilde S_N^h(y)-\widetilde x_N^h(y)]
\geq\Delta_{N+1}^h(y)+|\Delta_{N+1}^h(y)|=0,\\
&[\widetilde S_{N+1}^h(y)+\widetilde x_{N+1}^h(y)]
-[\widetilde S_N^h(y)+\widetilde x_N^h(y)]
\leq\Delta_{N+1}^h(y)-|\Delta_{N+1}^h(y)|\leq0.
\end{align*}
Finally, $\widetilde x_N^h(y)\to0$, so both bounding sequences
converge to $h(y;t,a)$ by \eqref{eq:image-h-remainder}.
The inequality $\Delta_{N+1}^h(y)\leq0$ also proves
$\widetilde S_N^h(y)\searrow h(y;t,a)$.
All arguments remain valid at $y=0$ and $y=a$.
\end{proof}

\section{Proof of Lemma \ref{lemma:acc:I}}
\label{proof:lemma:acc:I}
Note that the numerator in the definition of $A_1(\beta_t)$ is the distribution function of the Kolmogorov-Smirnov distribution evaluated at $\sqrt{\beta_t/2}$, thus it is increasing in $\beta_t$ \cite{Kolmogorov1933,devroye1981series}. Differentiate the denominator:
\begin{align*}
    \frac{\partial}{\partial \beta_t} &  \Big[\left(1-\exp\left\{-\beta_t \right\}\right) + \beta_t \sum_{n=1}^\infty \exp\left\{-\beta_t n^2\right\}\Big] \\
    &= \exp\{-\beta_t\} + \sum_{n=1}^\infty \exp\left\{-\beta_t n^2\right\} -\beta_t\sum_{n=1}^\infty n^2\exp\left\{-\beta_t n^2\right\} \\
    &= (2-\beta_t)\exp\{-\beta_t\} + \sum_{n=2}^\infty (1-\beta_t n^2) \exp\{-\beta_t n^2\} < 0 \\
&\hspace{15mm}\mbox{ for }\quad \beta_t > \beta_0 = 2.271664
\end{align*}
Thus the denominator is decreasing in $\beta_t$, hence the ratio $A_1(\beta_t)$ is increasing on $(\beta_0, \infty)$.

\section{Proof of Theorem \ref{thm:acc:II}}
\label{proof:thm:acc:II}

Before proving Theorem \ref{thm:acc:II} we state and prove a short lemma.
\begin{lemma}
\label{lem:short}
Consider a random variable $X$ and measurable functions $f,g:\RR\to \RR$ such that $f(X)$, $g(X)$, and $f(X)g(X)$ are integrable and
\begin{itemize}
    \item $g$ is non-decreasing;
    \item If $m=\EE(f(X))$, there exists $c$ in the support of $X$ such that
    $$
    f(x)\ge m\quad \mbox{ if }\ x\le c,\qquad
    f(x)\le m \quad \mbox{ if }\ x >c\ .
    $$
\end{itemize}
Then, $\Cov(f(X), g(X)) \le 0$.
\end{lemma}
\begin{proof}
Indeed, 
$$
\Cov(f(X), g(X)) = \EE[ (f(X)-m)(g(X) - \EE(g(X)))] = \EE[ (f(X) - m)(g(X) - g(c))]
$$
On the event ${X\le C}$ the first term is non-negative,the second term is non-positive, thus the product is non-positive. On the event $[X>c]$ the first term is non-positive, the second term is non-negative, thus the product is always non-positive, hence $\Cov(f(X), g(X))\le 0$.
\end{proof}
Returning to the proof of the theorem, we start with the Jacobi function identity \eqref{eq:jacobi1}, where we set $y=1/2$ and $\pi^2 z = \beta_t$, and find
$$
\begin{aligned}
1 + 2\sum_{n=1}^\infty (-1)^n\exp\{-\beta_t n^2\}
&= \sqrt{\frac{\pi}{\beta_t}} \sum_{n=-\infty}^\infty \exp\left\{ - \frac{\pi^2(2n+1)^2}{4\beta_t}\right\}\\
&= 2 \sqrt{\frac{\pi}{\beta_t}}\sum_{\substack{n=1\\ n\ \mathrm{odd}}}^\infty \exp\left\{-\frac{\pi^2 n^2}{4\beta_t}\right\}\ .
\end{aligned}
$$
As such, the acceptance ratio becomes
\begin{align*}
A_2(\beta_t) &= \frac
{\displaystyle 2 \sqrt{\frac{\pi}{\beta_t}}\sum_{\substack{n=1\\ n\ \mathrm{odd}}}^\infty \exp\left\{-\frac{\pi^2 n^2}{4\beta_t}\right\}}
{\displaystyle 4^{-1}\pi^{5/2}\beta_t^{-1/2}\sum_{n=1}^\infty n^2\exp\left\{-\frac{n^2\pi^2}{4\beta_t}\right\}}
=
\frac{8}{\pi^2}
\cdot
\frac{\displaystyle \sum_{\substack{n=1\\ n\ \mathrm{odd}}}^\infty \exp\left\{-\frac{\pi^2 n^2}{4\beta_t}\right\}}
{\displaystyle \sum_{n=1}^\infty n^2\exp\left\{-\frac{n^2\pi^2}{4\beta_t}\right\}}\\
&= \frac{8}{\pi^2}\cdot
\frac{\displaystyle \sum_{\substack{n=1\\ n\ \mathrm{odd}}}^\infty \exp\{-\alpha_t n^2\} }
{\displaystyle \sum_{n=1}^\infty n^2\exp\{-\alpha_t n^2\}}\ ,
\end{align*}
 where we recall that $\alpha_t = \pi^2/(4\beta_t)$.
Define
\[
N(\alpha):=\sum_{\substack{n\ge 1\\ n\ \mathrm{odd}}} e^{-\alpha n^2},\qquad
D(\alpha):=\sum_{n=1}^\infty n^2 e^{-\alpha n^2},\qquad
\Lambda(\alpha):=\frac{N(\alpha)}{D(\alpha)}.
\]
We drop the dependence on $t$, to simplify notation. We see that 
$
A_2(\beta_t)=\frac{8}{\pi^2}\,\Lambda\!\left(\frac{\pi^2}{4\beta_t}\right)
$, $\beta_t>0$.
Consider a discrete random variable $X\in \{1,2,\ldots\}$ with distribution
$$
\PP(X = n) = \frac{n^2\exp\{-\alpha n^2\}}{D(\alpha)} \qquad n\ge 1\ .
$$
Define $\phi(n):=\mathbf 1_{\{n\ \mathrm{odd}\}}/n^2$. Then
$$
\displaystyle 
\Lambda(\alpha)=\frac{N(\alpha)}{D(\alpha)}=\sum_{n\ge 1}\frac{\mathbf 1_{\{n\ \mathrm{odd}\}}}{n^2}\,\frac{n^2 \exp\{-\alpha n^2\} }{D(\alpha)}
=\mathbb E[\phi(X)]\ .
$$
We claim
\begin{equation}\label{eq:cov-identity}
\Lambda'(\alpha) = \frac{\partial}{\partial \alpha}\Lambda(\alpha)=-\operatorname{Cov}\!\big(\phi(X),X^2\big).
\end{equation}
Indeed, for $\alpha\ge \alpha_0>0$ the series $N(\alpha), D(\alpha)$ converge absolutely, so  we may differentiate termwise:
\[
\frac{\partial}{\partial \alpha}N(\alpha)= -\sum_{n\ge 1}\phi(n)\,n^2\cdot n^2 \exp\{-\alpha n^2\},\qquad
\frac{\partial}{\partial \alpha}D(\alpha)= -\sum_{n\ge 1}n^4 \exp\{-\alpha n^2\}.
\]
The quotient rule gives
\[
\Lambda'(\alpha)=\frac{N'(\alpha)D(\alpha)-N(\alpha)D'(\alpha)}{D(\alpha)^2}
=-\,\mathbb E\big[\phi(X)X^2\big]
+\mathbb E[\phi(X)]\,\mathbb E[X^2],
\]
which is \eqref{eq:cov-identity}. It remains to show that 
$$
\Cov(\phi(X), X^2) \le 0\ .
$$
For this, we use the lemma above with $f(n) = \phi(n)$, $g(n) = n^2$, and $c=1$. It is clear that $f(X)$ and $g(X)$ are integrable, $g$ is non-decreasing and 
$$
m_\alpha = \EE(\phi(X)) = \frac{N(\alpha)}{D(\alpha)} = \Lambda(\alpha) < \infty\ .
$$
Note that $\phi(1) = 1$ and $\phi(n) \le 1/9$ for $n\ge 2$. It suffices to prove that 
$$
\frac{1}{9} < m_\alpha < 1\ .
$$
The upper bound is immediate because $\phi(1) = 1$, $\phi(n) < 1$ for $n\ge 2$ and $\PP(X\ge 2)>0$. For the lower bound, note that $N(\alpha) > e^{-\alpha}$ and thus
$$
m_\alpha \ge \frac{e^{-\alpha}}{D(\alpha)}
$$
For $\alpha \ge 1$,
\begin{align*}
\frac{D(\alpha)}{e^{-\alpha}} &= 1 + \sum_{n=2}^\infty n^2e^{-\alpha(n^2-1)} 
 <  1 + \sum_{n=2}^\infty n^2e^{-n}\\
&<  1 + \sum_{n=2}^\infty n^22^{-n} = 1 + \left(6-\frac{1}{2}\right) = \frac{13}{2}\ .
\end{align*}  
Thus, when $\alpha \ge 1$
$$
m_\alpha \ge \frac{e^{-\alpha}}{D(\alpha)} > \frac{2}{13} > \frac{1}{9}\ ,
$$
as desired. Thus, $\Lambda$ is increasing in $\alpha_t$ and, equivalently, decreasing in $\beta_t$, over $\alpha_t\ge 1$, and thus, including when $\alpha_t\ge\alpha_0>1$, which is equivalent with $\beta_t \le \beta_0$.
\end{appendices}

\backmatter
\section*{Statements and Declarations}
\subsection*{Funding}
Nothing to declare.

\subsection*{Competing interests}
The authors declare that they have no conflict of interest.

\subsection*{Author contributions}
Both authors contributed equally to the conception of the study and the theoretical analysis. Radu Herbei drafted the manuscript and performed the computer simulations. Both authors contributed equally to the final editing of the manuscript.

\subsection*{Data availability}
The simulation data supporting the figures can be regenerated using the Python scripts in the project repository. The repository also provides the code for reproducing the acceptance-rate and timing experiment.

\subsection*{Code availability}
Python code implementing the sampling algorithms and reproducing the figures and computational experiment is available at~\url{https://github.com/herbei/Brownian_Extrapolation_Interpolation}.

\subsection*{Use of AI tools}
ChatGPT was used to assist with debugging Python code and improving grammar and language.

\bibliography{references}
\end{document}